\documentclass[11pt,a4paper]{amsart}
\usepackage{amsmath,amsthm,amsfonts,graphicx,amssymb,amscd}
\usepackage{graphicx}
\usepackage{hyperref}
\newtheorem{thm}{Theorem}[section]
\newtheorem*{jthm}{Theorem}
\newtheorem{conj}{Conjecture}

\newtheorem{ques}[thm]{Question}
\newtheorem{cor}[thm]{Corollary}
\newtheorem{lem}[thm]{Lemma}
\newtheorem{prop}[thm]{Proposition}
\theoremstyle{definition}
\newtheorem{defn}[thm]{Definition}

\begin{document}

\title{Meromorphic quadratic differentials with half-plane structures}
\author{Subhojoy Gupta}
\address{Center for Quantum Geometry of Moduli Spaces, Ny Munkegade 118, DK 8000 Aarhus C, Denmark.}
\email{sgupta@qgm.au.dk}
\date{Feb 17, 2012}

\begin{abstract}We prove the existence of  ``half-plane differentials"  with prescribed local data on any Riemann surface. These are meromorphic quadratic differentials with higher-order poles which have an associated singular flat metric isometric to a collection of euclidean half-planes glued by an interval-exchange map on their boundaries. The local data is associated with the poles and consists of the integer order, a non-negative real residue, and a positive real leading order term. This generalizes a result of Strebel for differentials with double-order poles, and associates metric spines with the Riemann surface. 
\end{abstract}

\maketitle

\tableofcontents

\section{Introduction}

A holomorphic quadratic differential on a Riemann surface defines a conformal metric that is flat with conical singularities, together with pair of \textit{horizontal} and \textit{vertical} measured foliations, and this singular-flat geometry is intimately related to quasiconformal mappings and the geometry of Teichm\"{u}ller space $\mathcal{T}_g$  in the Teichm\"{u}ller metric. In this paper we consider certain \textit{infinite area} singular flat surfaces corresponding to (non-integrable) meromorphic quadratic differentials that arise as geometric limits along Teichm\"{u}ller geodesic rays.\\

A \textit{half-plane surface} is a singular flat surface that is obtained by taking a finite partition of the boundaries of a collection of euclidean half-planes and gluing by an interval-exchange map (see \S4 for examples.) We shall only exclude the possibility that the two infinite-length boundary intervals of the same half-plane are identified. Such a surface can be thought of as a Riemann surface with punctures $p_1,p_2,\ldots p_n$ corresponding to the ends of the surface, equipped with a quadratic differential (restricting to $dz^2$ on the half-planes) that is holomorphic away from the punctures. This \textit{half-plane differential} has a pole of order $n_j\geq 4$  at  $p_j$ (for $1\leq j\leq n$) and one can associate with it the data of a non-negative real \textit{residue} $a_j$ (this is always zero when $n_j$ is odd) and, in a given choice of local coordinates, a positive real \textit{leading order term} $c_j$, which is the top coefficient in a local series expansion of the differential. In the singular flat metric, a neighborhood of $p_j$ is isometric to a ``planar end" comprising $n_j-2$ half-planes glued cyclically along their boundaries (see Definition \ref{defn:pend}). The residue $a_j$ then corresponds to the metric holonomy around the puncture and $c_j$ gives the ``scale" of this planar end relative to the others. We  provide precise definitions  in \S3.\\

In this article we show that one can get \textit{any} Riemann surface with $n$ punctures by the above construction, and furthermore can arbitrarily prescribe the data of order, residue and leading order terms:

\begin{center}\label{eq:data}
$\mathcal{D} = \{ (n_j, a_j,c_j)| n_j\in \mathbb{N}, n_j\geq 4, c_j\in \mathbb{R}^{+}, a_j\in \mathbb{R}_{\geq 0}, a_j=0$ for $n_j$ odd.$\}$
\end{center}
associated with the set of punctures:

\begin{thm}\label{thm:main} Let $\Sigma$ be a Riemann surface with a set $P$ of $n$ marked points and with a choice of local coordinates around each. Then for any data $\mathcal{D}$ as above there is a corresponding half-plane surface $\Sigma_\mathcal{D}$ and a conformal homeomorphism \begin{equation*}
g:\Sigma\setminus P \to \Sigma_\mathcal{D}
\end{equation*}
that is homotopic to the identity map.\\
(The only exception is for the Riemann sphere with one marked point with a pole of order $4$, in which case the residue must equal zero.)
\end{thm}

Note that it is not hard to show the existence of \textit{some} meromorphic quadratic differential which has local data given by $\mathcal{D}$ at the poles using the Riemann-Roch theorem, but half-plane differentials are a special subclass that satisfy the \textit{global} requirement of having a ``half-plane structure" as described.\\

This result can be thought of as a generalization of the following theorem of Strebel (\cite{Streb}):

 \begin{jthm}[Strebel]\label{thm:streb} Let $\Sigma$ be a Riemann surface of genus $g$, and $P = \{p_1,p_2,\ldots p_n\}$ be marked points on $\Sigma$ such that $2g-2+n>0$, and   $(a_1,a_2,\ldots,a_n)$ a tuple of positive reals. Then there exists a meromorphic quadratic differential $q$ on $\Sigma$ with poles at $\mathcal{P}$ of order $2$ and  residues $a_1,\ldots a_n$, such that all horizontal leaves (except the critical trajectories) are closed and foliate punctured disks around $\mathcal{P}$.
\end{jthm}

The corresponding singular flat metric for such a ``Strebel differential" with poles of order two comprises a collection of half-infinite cylinders glued by an interval-exchange on their boundaries, and has a metric spine (sometimes called the ``ribbon graph" or ``fat graph"). Strebel also showed that the quadratic differential $q$ above - and hence this metric spine - is unique, which yields useful combinatorial descriptions of Teichm\"{u}ller space (see, for example, \cite{HarZag}, \cite{Kont}). However, a  corresponding uniqueness statement for Theorem \ref{thm:main}  is not known, and conjecturally holds for poles of order $4$ (for a discussion, see \S13.2).\\

The proof of Theorem \ref{thm:main} uses the well-known result of Jenkins and Strebel that associates a holomorphic quadratic differential to a collection of curves on a \textit{compact} Riemann surface. The idea is to consider a compact exhaustion of $\Sigma\setminus P$ and produce a corresponding sequence of half-plane surfaces that shall have both $\Sigma\setminus P$ and $\Sigma_{\mathcal{D}}$ as its conformal limit. The main technical work is to obtain enough geometric control for this sequence,  and extract these limits by building appropriate sequences of quasiconformal maps.  We give a more detailed outline in \S5, and carry out the proof in \S6-12. In \S13 we conclude with some applications and further questions. In particular, the connection of half-plane surfaces with conformal limits of grafting and Teichm\"{u}ller rays will appear in a subsequent paper (\cite{Gup25}).\\

\medskip

\textbf{Acknowledgements.} 
This paper arose from work of the author as a graduate student at Yale, and he wishes to thank his thesis advisor Yair Minsky for his generous help. The author also wishes to thank Michael Wolf for numerous helpful discussions, Kingshook Biswas and Enrico Le Donne for useful conversations, and the support of the Danish National Research Foundation center of Excellence, Center for Quantum Geometry of Moduli Spaces (QGM) where this work was completed.

\section{Background}

In  this section we review some background on meromorphic quadratic differentials and the geometry they induce on a Riemann surface.\\

Most of this is standard (we refer to \cite{Streb}), but we point the reader to our notion of a  ``planar end" (Definition \ref{defn:pend}), which provides a metric model for an end of a half-plane surface, or more generally for the quadratic differential metric around higher-order poles.

\subsection{Quadratic differential space}

A closed genus-$g$ Riemann surface $\Sigma_g$ admits no non-constant holomorphic functions, but carries a finite-dimensional vector space of holomorphic one-forms, or more general holomorphic differentials, which are holomorphic sections of powers of the canonical line-bundle $K_\Sigma$.\\

A \textit{quadratic differential} $q$ on $\Sigma_g$ is a section of $K_\Sigma \otimes K_\Sigma$, a differential of type $(2,0)$ locally of the form $q(z)dz^2$. It is said to be \textit{holomorphic} (or \textit{meromorphic}) when $q(z)$ is holomorphic (or meromorphic).\\

A \textit{zero} or a \textit{pole} of a holomorphic quadratic differential is a point $p$ where in a local chart sending $p$ to $0$ we have $\phi(z) =z^n\psi(z)$ or $\phi(z) = \frac{1}{z^n}\psi(z)$ respectively, where $\psi(z)\neq 0$ and the integer $n\geq 1$ is the \textit{order} of the zero or pole. The following is a well-known fact:

\begin{lem}\label{lem:gb} If there are $M$ zeroes of orders $n_1,n_2,\ldots n_M$, and $N$ poles of orders  $k_1,k_2,\ldots k_N$, then
$\sum\limits_{i=1}^M n_i - \sum\limits_{i=1}^N k_i  = 4g-4$.
\end{lem}

Let $Q_{g,k_1,k_2,\ldots k_N}$ be the space of meromorphic quadratic differentials on $\Sigma_g$ with poles of order less than or equal to $k_1,k_2,\ldots k_N$ at a collection of $N$ marked points, and $\widehat{Q}_{g,k_1,k_2,\ldots k_N}$ be the subset of such differentials with poles of orders \textit{exactly} $k_1,k_2,\ldots k_N$.
The former is a finite dimensional vector space, whose dimension can be computed using the Riemann-Roch theorem.\\



In fact the vector space $Q(X)$ of holomorphic quadratic differentials on a Riemann surface $X$ has complex dimension exactly $3g-3$ (see for example \cite{FarKra}).

\subsection{Quadratic differential metrics}

A holomorphic quadratic differential $q\in Q(X)$ defines a conformal metric (also called the $q$-metric) given in local coordinates by $\lvert q(z)\rvert \lvert dz\rvert^2$ which has Gaussian curvature zero wherever $q(z)\neq 0$.\\

At the points where $q(z)=0$ (finitely many by Lemma \ref{lem:gb}) there is a conical singularity of angle $(n+2)\pi$ where $n$ is the order of the zero, and locally the singular flat metric looks like a collection of $n$ rectangles glued around the singularity (see \S7 of \cite{Streb}).\\

\begin{figure}[h]\label{fig:zero}
  \centering
  \includegraphics[scale=0.85]{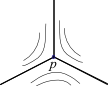}\\
  \caption{A simple zero $p$ is a $3\pi$ cone-point. }
\end{figure}

One way to see this is to change to coordinates where $q = d\zeta^2$ by the conformal map
\begin{equation}\label{eq:choc}
z\mapsto \displaystyle\int\limits_p^z \sqrt{q(z)}dz
\end{equation}
which gives a branched covering of the $\zeta$-plane when $p$ is a zero of $q$.

\subsubsection*{Lengths, area} The \textit{horizontal length} in the $q$-metric of an arc $\gamma$ is defined to be
\begin{equation*}
\lvert \gamma\rvert_h = Re \displaystyle\int\limits_{\gamma} \sqrt{q(z)} d\vert z\rvert
\end{equation*}
and the \textit{vertical length} is the corresponding imaginary part. The total length in the $q$-metric is $\left(\lvert \gamma\rvert_h^2 + \lvert \gamma\rvert_v^2\right)^{1/2}$.\\ 

The $L^1$-norm of a quadratic differential gives the \textit{area} in the $q$-metric:
\begin{equation}\label{eq:area}
Area_q (X) =  \displaystyle\int\limits_X \lvert \phi(z)\rvert dzd\bar{z}
\end{equation}

\subsection{Measured foliations}
A holomorphic quadratic differential $q\in Q(X)$ determines a \textit{horizontal foliation} on $X$ which we denote by $\mathcal{F}_h(q)$, obtained by integrating the line field of vectors $\pm v$ where the quadratic differential is real and positive, that is $q(v,v)\geq 0$. Similarly, there is a \textit{vertical foliation} $\mathcal{F}_v(q)$ consisting of integral curves of directions where $q$ is real and negative.\\

These foliations can be thought of as the pullback by the map (\ref{eq:choc}) of the horizontal and vertical lines in the $\zeta$-plane. These foliations are also \textit{measured}: the measure of an arc transverse to $\mathcal{F}_h$ is given by its \textit{vertical} length, and the transverse measure for $\mathcal{F}_v$ is given by horizontal lengths. Such a measure is invariant by isotopy of the arc if it remains transverse with endpoints on leaves.\\

Let $\mathcal{MF}$ be the space of singular foliations on a surface equipped with a transverse measure, upto isotopy and Whitehead equivalence.

\begin{jthm}[Hubbard-Masur \cite{HM}]\label{thm:HM} Fix a Riemann surface $X$. Then any $\mathcal{F}\in \mathcal{MF}$ is the horizontal foliation of a unique holomorphic quadratic differential on $X$.
\end{jthm}

Theorem \ref{thm:main} of this paper can be thought of as a step towards a generalization to a \textit{non-compact} version of the above result. In this paper we shall use a special case independently proved in \cite{JS1} and \cite{JS2} (See \cite{Wolf1} for a proof using harmonic maps to metric graphs.)

\begin{thm}[Jenkins-Strebel]\label{thm:js}
Let $\gamma_1,\gamma_2,\ldots \gamma_n$ be disjoint homotopy classes of curves on a Riemann surface $\Sigma$. Then there exists a holomorphic quadratic differential $q$ whose horizontal foliation $\mathcal{F}_h(q)$ consists of closed leaves foliating $n$ cylinders with core curves in those homotopy classes. Moreover, one can prescribe the heights, or equivalently the circumferences, of the $n$ metric cylinders to be any $n$-tuple $(c_1,c_2,\ldots,c_n)$ of positive reals.
\end{thm}

\begin{figure}[h]
  \centering
  \includegraphics[scale=0.55]{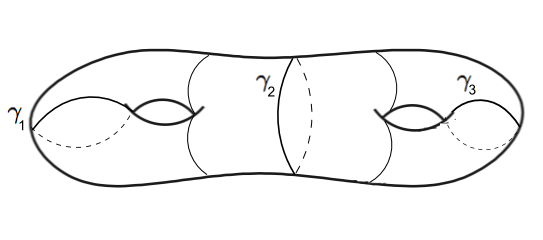}\\
  \caption{In Theorem \ref{thm:js} the surface in the $q$-metric is composed of metric cylinders with core curves the given homotopy classes of curves. }
\end{figure}

\subsection*{Critical segments} A \textit{critical segment} is a horizontal arc between two critical points (zeroes or poles) of the quadratic differential, and a critical arc between two zeroes is called a \textit{saddle-connection}. Their number is always finite since by the theorem of Gauss-Bonnet (see also Theorem 14.2.1 in \cite{Streb}) there cannot be more than one saddle-connection between a pair of zeroes.

\subsection{Metric structure at poles}

Most of the preceding discussion also holds when the quadratic differential $q$ is \textit{meromorphic}, with finitely many poles: namely, one has an associated singular flat $q$-metric, together with horizontal and vertical foliations. \\

A meromorphic quadratic differential with a finite $q$-area can have poles of order at most $1$ (see \cite{Streb}). At such a pole, there is a conical singularity of angle $\pi$, and the singular flat metric has a ``fold" (see figure, also \S7 of \cite{Streb}).\\

A pole of higher order $n>1$ is at an infinite distance in the $q$-metric. Any such pole $p$ has an associated \textit{analytic residue} which is defined to be 
\begin{equation}\label{eq:ares}
Res_q(p) = \displaystyle\int\limits_\gamma \sqrt{q}
\end{equation}
where $\gamma$ is any simple closed curve homotopic into $p$ (this is independent of choice of $\gamma$ since $q$ is holomorphic away from $p$).\\

\begin{figure}\label{fig:poles}
  \centering
  \includegraphics[scale=0.85]{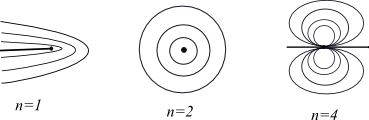}\\
  \caption{Local picture with horizontal leaves at a pole of order $n$.}
  \end{figure}
  
For the rest of this article, we shall consider only higher order poles which have a \textit{real} analytic residue (in \cite{Streb} this property is referred to as $q$ having a \textit{vanishing logarithmic term}). This rules out any ``spiralling" behaviour of the horizontal foliation at the poles. We also have an explicit description of the singular flat metric around the pole $p$ (Theorems \ref{thm:order2} and \ref{thm:streb}) which can be culled from \cite{Streb} (see \S7 of the book for a discussion).

\begin{thm}\label{thm:order2}
Let $p$ on $\Sigma$ be a pole of order $2$ with a positive real analytic residue $C$. Then there is a neighborhood $U$ of $p$ such that in the $q$-metric  $U\setminus p$ is isometric to a half-infinite euclidean cylinder with circumference $2\pi C$.
\end{thm}
\begin{proof}
A neighbourhood $U$ of $p$ has a conformal chart  to $\mathbb{D}$ taking $p$ to $0$ the quadratic differential takes the form
\begin{equation*}
-\frac{C^2}{z^2} dz^2
\end{equation*}
The change of coordinates (\ref{eq:choc}) is given by a logarithm map $\zeta = iC\ln z$ that pulls back the euclidean $\zeta-$plane to a half-infinite cylinder. We leave these verifications to the reader. \end{proof}

\subsection{Planar ends}

We now introduce the notion of a \textit{planar end} and its \textit{metric residue}. \\
In the following definition we identify the boundary of a euclidean half-plane with $\mathbb{R}$.

\begin{defn}[Planar-end]\label{defn:pend}
Let $\{H_i\}$ for $1\leq i\leq n$ be a cyclically ordered collection of half-planes with rectangular ``notches" obtained by deleting, from each, a rectangle of horizontal and vertical sides adjoining the boundary, with the boundary segment having end-points $x_i$ and $y_i$, where $x_i<y_i$. A \textit{planar end} is obtained by gluing the interval $[y_i,\infty)$ on $\partial H_i$ with $(-\infty, x_{i+1}]$ on $H_{i+1}$ by an orientation-reversing isometry. Such a surface is homeomorphic to a punctured disk.
\end{defn}

For an example, see Figure 4.

\begin{defn}[Metric residue]\label{defn:metres}
Let the  half-plane differential $q$ at a pole $p_j$ have a planar end as above (where the number of half-planes $n$ equals $ n_j -2$). Then the residue $a_j$ of  $q$ at $p_j$  is defined to be the alternating sum $\sum\limits_{i=1}^n (-1)^{i+1}(y_i - x_i)$. Note that there is an ambiguity of sign because of the cyclic ordering in the alternating sum, and we resolve this by choosing a starting index that ensures a positive sum.
\end{defn}

\begin{figure}[h]
  \centering
  \includegraphics[scale=0.45]{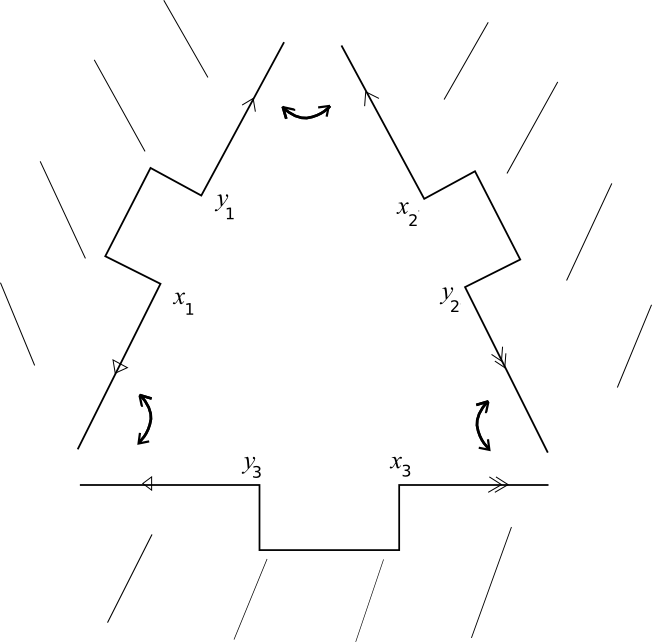}\\
  \caption{A planar-end obtained from $3$ ``notched" half-planes (see Definition \ref{defn:pend}). }
  \end{figure}

\begin{thm}\label{thm:streb}
Let $p$ on $\Sigma$ be a pole of order $n>2$ with a real analytic residue $C$. Then there is a neighborhood $U$ of $p$ such that in the $q$-metric $U\setminus p$ is isometric to a planar-end surface with $(n-2)$ half-planes and metric residue equal to $C$.
\end{thm}
\begin{proof}
We provide only the proof for the statement regarding the metric residue, as the proof of the planar-end structure appears in \cite{Streb} (see \S10.4 of the book).\\

Let $P$ be a polygon made from the boundaries of the rectangular ``notches" of each half-plane of the planar end (see Definition \ref{defn:pend}) oriented counter-clockwise. Then $P$ is a polygon enclosing $p$ contained in $U$, and consists of alternating horizontal and vertical segments, with the horizontal edges having lengths $b_i-a_i$ for $1\leq i\leq n$.\\

We first note that
\begin{equation}\label{eq:asum}
\pm \int\limits_{P} \sqrt q = \displaystyle\sum\limits_{i=1}^n (-1)^{i+1}(b_i-a_i)
\end{equation}

This follows from the fact that by a change of coordinates (see (\ref{eq:choc})) the integral over three successive edges (horizontal-vertical-horizontal) equals integrating the form $d\zeta$ on the complex ($\zeta$-)plane over a horizontal  edge that goes from the right to left on the upper half-plane followed by one over a vertical edge followed by one over a horizontal edge that goes from left to right in the lower half-plane. Hence the integral over the horizontal sides picks up the horizontal lengths but the sign switches  over the two successive horizontal edges. Meanwhile the integral over the vertical sides contribute to the imaginary part, but they cancel out since two successive vertical segments on the upper ($\zeta$)-half-plane are of equal length but of opposite orientation.\\

The left hand side of (\ref{eq:asum}) is equal to the real analytic residue $C$ (upto sign) by definition (\ref{eq:ares}), and the right hand side is equal to the metric residue of the planar end as defined in Definition \ref{defn:pend}, and the proof is complete.
\end{proof}

\begin{table}\label{table:tab}

\begin{center}
\begin{tabular}{|l|l|}
  \hline
  \textbf{Analytic notions} & \textbf{Metric notions} \\ \hline
   $L^1$-norm  & Area \\
  \hline
   At least one non-simple pole & Infinite area \\
  \hline
   Zero of order $n$ & Cone point of angle $(n+2)\pi$ \\
  \hline
   Simple (order $1$) pole  & Cone point of angle $\pi$\\
  \hline
   Order $2$ pole  & Half-infinite cylinder \\
  \hline
   Order $n>2$ pole  & Planar end \\
  \hline
   Analytic residue   & Metric residue \\
  \hline

\end{tabular}
\end{center}
\caption{A glossary of the correspondence between analytic and metric properties of a meromorphic quadratic differential.}
\end{table}

\section{Preliminaries}

In this section we introduce some of the terminology and observations used in this paper.

\subsection{Half-plane differentials.}

The following definition was already mentioned in \S1:

\begin{defn}[Half-plane surface]\label{defn:hp}
Let $\{H_i\}_{1\leq i\leq N}$ be a collection of $N\geq 2$ euclidean half planes and let $\mathcal{I}$ be a finite partition into sub-intervals of the boundaries of these half-planes. A \textit{half-plane surface} $\Sigma$ is a complete singular flat surface obtained by gluings by (oriented) isometries amongst intervals from $\mathcal{I}$.
\end{defn}

Such a half-plane surface has a number of planar ends as in Definition \ref{defn:pend}, and is equipped with a meromorphic quadratic differential $q$ (the \textit{half-plane differential})  that restricts to $dz^2$ in the usual coordinates on each half-plane.

\subsection{Residue}
The residue $a_j$ associated with a puncture (or end) of a half-plane surface has both an metric  definition (Definition \ref{defn:metres}) and the following analytic definition (see (\ref{eq:ares}) ):

\begin{defn}[Analytic residue] The residue $a_j$ of the half-plane differential $q$ at a pole $p_j$ is defined to be the absolute value of the integral
\begin{equation*}
a_j = \int\limits_{\gamma_j} \sqrt{q}
\end{equation*}
where $\gamma_j$ is a simple closed curve enclosing $p_j$ and contained in a chart where one can define $\pm\sqrt{q}$. This gives a positive real number (see Theorem \ref{thm:streb}) and in particular, equals the metric residue of the corresponding planar end.
\end{defn}

\subsection{Planar ends and truncations}

A planar end (Definition \ref{defn:pend}) can be thought of as a neighborhood of $\infty$ on $\mathbb{C}$ in the metric induced by the restriction of a ``standard" holomorphic quadratic differential $\phi$:
\begin{equation}\label{eq:phi}
\phi = \phi(z)dz^2 = \left(z^{n-2} + iaz^{n/2 -2}\right)dz^2
\end{equation}
where $a$ is the (positive, real) residue at the pole at $0$ and $n\geq 4$ is even. See \textit{Example 2} in \S4.1 for details. When $n$ is odd, the residue is necessarily zero, and the metric is induced by the differential $\phi = \phi(z)dz^2 = z^{n-2}dz^2$.\\
By inversion ($w= 1/z$), this can be thought of as the metric induced by the restriction of a meromorphic quadratic differential (which we also denote by $\phi$) to a neighborhood of $0\in\mathbb{C}$:
\begin{equation}\label{eq:stan}
\left(\frac{1}{w^{n+2}} + \frac{ia}{w^{n/2 +2}}\right)dw^2
\end{equation}
when $n$ is even and $\left(\frac{1}{w^{n+2}}\right) dw^2$ when $n$ is odd.

\begin{defn}[$\mathcal{P}_H$]\label{defn:trunc} For a planar end with residue $a$, the \textit{truncation at height $H$} denoted by $\mathcal{P}_H$, is when the missing ``notches" are rectangles of horizontal width $H$ and vertical heights $H/2$, except one rectangle of horizontal width $H+a$.  The boundary of the planar $\mathcal{P}_H$ is then a polygon of alternating horizontal and vertical sides, each of length $H$ except one horizontal side of length $H+a$. (This is then compatible with the metric residue being $a$ - see Definition \ref{defn:metres}.)
\end{defn}

\textit{Remark.} Any planar end has a  truncation at height $H$, for sufficiently large $H$.\\

\begin{figure}[h]
  \centering
  \includegraphics[scale=0.6]{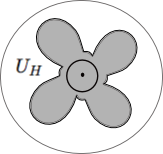}\\
  \caption{A truncation of a planar end can be conformally identified with a neighborhood $U_H$ of $0\in \mathbb{C}$ (shown shaded). Lemma \ref{lem:uh} gives estimates on its dimensions. }
  \end{figure}

By the previous discussion we can identify  $\mathcal{P}_H\cup \infty$ with a neighborhood $U_H$ of $0\in \mathbb{C}$ , via a map taking $\infty$ to $0$, which is an isometry with respect to the singular flat metric on the planar end and the ${\phi}$-metric on $U_H$, and is hence conformal. The following estimates about this simply connected domain in $\mathbb{C}$ will be useful later:

\begin{lem}\label{lem:uh} There exist universal constants $D_1, D_2>0$ such that for all sufficiently large $H$, we have:
\begin{equation}\label{eq:uh}
\frac{D_1}{H^{2/n}} \leq dist(0,\partial U_H)  \leq \frac{D_2}{H^{2/n}}
\end{equation}
where $n$ is the number of half-planes for the planar end.\\ 

(In this paper, $dist(q, K) = \inf\limits_{p\in K} \left| p-q\right|$ where $q\in \mathbb{C}$ and $K\subset \mathbb{C}$ is a compact set.) 
\end{lem}
\begin{proof}
Observe that given a planar end, one can circumscribe a circle of circumference $C_2H$ and inscribe one of circumference $C_1H$ around $\partial \mathcal{P}_H$, where $0<C_1<1<C_2$ are constants that depend only on $n$ (for sufficiently large $H$, the effect of the fixed $a$ is negligible).\\
\begin{figure}[h]
  \centering
  \includegraphics[scale=0.5]{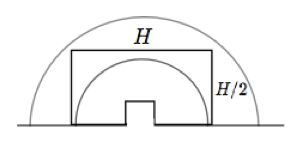}\\
  \end{figure}

Now the circumference of the boundary circle $\partial B_r$ in the metric induced by $q$ (see  (\ref{eq:stan}) is:
\begin{equation*}
C(r) = \int\limits_{\partial B_r} \lvert \sqrt{q}\rvert   = \int\limits_{\lvert w\rvert = r} \Big\lvert\frac{1}{w^{n/2+1}} + \frac{ia}{2w}\Big\rvert dw  \approx   \int\limits_{\lvert w\rvert = r} \Big\lvert\frac{1}{w^{n/2+1}} \Big\rvert dw  = O\left(\frac{1}{r^{n/2}}\right)
\end{equation*}
for sufficiently small $r$.\\

The previous observation together with this calculation then implies that $U_H$ is contained within two boundary circles which yield (\ref{eq:uh}). (Here $D_1,D_2$ depend on $C_1,C_2$.)
\end{proof}

\begin{cor}\label{cor:uhmap}
The conformal map $\phi_H:U_H\to \mathbb{D}$ that takes $0$ to $0$, satisfies
\begin{equation}\label{eq:uhbd}
\frac{1}{4D_2}\leq \frac{\left|\phi_H^\prime(0)\right|}{H^{2/n}}\leq \frac{1}{D_1}
\end{equation}
for all $H>0$.
\end{cor}
\begin{proof}  Note that such a conformal map is determined uniquely upto rotation (fixing $0$), but this does not change the magnitude of the derivative at $0$.\\
For any conformal map $f:\mathbb{D}\to \mathbb{C}$ the following holds for any $z\in \mathbb{D}$ (see Corollary 1.4 of \cite{Pom}):
\begin{equation*}
\frac{1}{4} \left(1 - \left|z\right|^2\right)\left|f^\prime(z)\right| \leq dist (f(z), \partial f(\mathbb{D})) \leq  \left(1 - \left|z\right|^2\right)\left|f^\prime(z)\right|
\end{equation*}
We apply this to the conformal map $f = \phi_H^{-1}: \mathbb{D} \to U_H\subset \mathbb{C}$ and $z=0$. We get by rearranging and using that $f(0)=0$ that
\begin{equation}\label{eq:shk2}
dist(0, \partial f(\mathbb{D})) \leq \left|f^\prime(0)\right| \leq 4 dist(0, \partial f(\mathbb{D}))
\end{equation}
By the previous lemma, and the fact that $\phi_H^\prime(0) = 1/f^\prime(0)$, (\ref{eq:uhbd}) now follows.
\end{proof}

We also note the following monotonicity:
\begin{lem}\label{lem:mon} Let $\phi_H:U_H\to \mathbb{D}$ be the conformal map preserving $0$. Then the derivative $\left|\phi^\prime_H(0)\right|$ is strictly increasing with $H$.
\end{lem}
\begin{proof} For $\hat{H}>H$ we have the strict inclusions $\mathcal{P}_H\subset \mathcal{P}_{\hat{H}}$ and $U_{\hat{H}} \subset U_H$. Hence $\phi_H\circ \phi_{\hat{H}}^{-1}:\mathbb{D}\to \mathbb{D}$ is well-defined, and the lemma follows from an application of the Schwarz lemma.
\end{proof}

\subsection{Leading order term}

\begin{defn} In a choice of local coordinates $z$ around the pole $p$, any meromorphic quadratic differential $q$ has a local expression:
\begin{equation*}
q = \left(\frac{a_n}{z^n} + \frac{a_{n-1}}{z^{n-1}} +\cdots \frac{a_{1}}{z} + a_0 +  \cdots\right)dz^2
\end{equation*}
and we define the \textit{leading order term} at the pole to be  the positive real number $c_j = \left| a_n\right|$.
\end{defn}

\textit{Remarks.} 1. A pole at $p$ of a half-plane differential $q$ has a neighborhood $U$ that is isometric to a planar end as in Definition \ref{defn:pend}. The leading order term of a half-plane differential at a pole determines, roughly speaking, the relative scale of the conformal disk $U$ is on the Riemann surface.\\

2. We shall sometimes refer to the leading order term of $q$ at $p$  \textit{with respect to $U$}, where $U$ is simply-connected domain containing $p$. This means that the coordinate chart that we consider is the conformal  (Riemann) map $\phi:U\to \mathbb{D}$ that takes $p$ to $0$. Note that by the above definition, the leading order term is independent of the choice of such a $\phi$ (rotation does not change the magnitude).

\begin{lem}[Pullbacks and leading order terms]\label{lem:pullb}Let $f:\mathbb{D}\to \mathbb{C}$ be a univalent conformal map such that $f(0)=0$ and let $q$ be a meromorphic quadratic differential on $\mathbb{C}$ having the local expression
\begin{equation*}
q(z)dz^2 = \left(\frac{a_n}{z^n} + \frac{a_{n-1}}{z^{n-1}} +\cdots \frac{a_{1}}{z} + a_0 +  \cdots\right)dz^2
\end{equation*}
in the usual $z$-coordinates. Then the pullback quadratic differential $f^\ast q$ on $\mathbb{D}$ has leading order term equal to $\left| f^\prime(0)\right|^{2-n}a_n$ at the pole at $0$.
\end{lem}
\begin{proof}
By the usual transformation law for change of coordinates, the local expression for the pullback differential is $q\circ f(z) f^\prime(z)^2$. The lemma follows by a calculation involving a series expansion, using that
\begin{center}
$f(z) = f^\prime(0)z + b_2z^2 +O(z^3)$.
\end{center} \end{proof}

\section{Examples}

\subsection{Explicit hpd's in $\hat{\mathbb{C}}$} A meromorphic quadratic differential on $\hat{\mathbb{C}} = \mathbb{C} \cup \{\infty\}$ can be expressed as $q(z)dz^2$ in the affine chart $\mathbb{C}$, where $q(z)$ is a meromorphic function. The following are examples of  such functions which yield a half-plane differential (hpd).\\

\textit{Example 0.} The quadratic differential $dz^2$ has a pole of order $4$ at $\infty$, and it induces the usual euclidean metric on the plane.\\

\textit{Example 1.} The quadratic differential $z^ndz^2$ for $n\geq 1$ has a pole of order $n+4$ at infinity, with analytic residue (and metric residue) equal to zero. The singular flat metric has $n+2$ half-planes glued around the origin.\\

\textit{Example 2.} The quadratic differential $\phi$ (see (\ref{eq:phi}))
\begin{equation*}
\left(z^{n-2} + iaz^{n/2 -2}\right)dz^2
\end{equation*}
 on $\mathbb{C}$ for an even integer $n\geq 4$ and some real number $a>0$ has $n/2+1$ zeroes, a pole of order $n+2$ at infinity, and a connected critical graph with a metric residue $\pi a$ (see figure, and \cite{HM} and \cite{Wan} for details).\\

\begin{figure}[h!]
  \centering
  \includegraphics[scale=0.7]{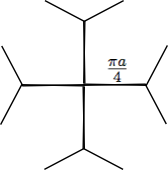}\\
  \caption{Partial picture of the critical graph when $n=8$ in \textit{Example 2}. The finite-length saddle connections have length $\pi a/4$. }
  \end{figure}
  
 In local $w$-coordinates at infinity obtained by inversion, the quadratic differential has the standard  form (\ref{eq:stan}) as in \S3.
The analytic residue can thus be computed in these coordinates as follows:
\begin{equation*}
Res_q(0) = \displaystyle\int\limits_\gamma \sqrt{q} =  \displaystyle\int\limits_\gamma \left(\frac{1}{w^{n/2+1}} + \frac{ia}{2w}\right)dw = \pi a
\end{equation*}
which is equal to the metric residue.

\subsection{Other hpd's in  $\hat{\mathbb{C}}$}

Since there is only a unique Riemann surface conformally equivalent to $\mathbb{C}$, it is easy to construct  half-plane differentials on the Riemann sphere:

 \begin{figure}[h]
  \centering
  \includegraphics[scale=0.75]{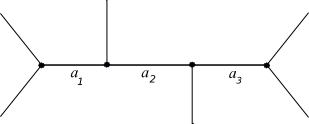}\\
  \caption{Attaching half-planes along this metric tree produces an hpd with a pole of order $8$ and residue $\left|2a_1-2a_3\right|$. }
  \end{figure}

\subsubsection*{Single-poled} Take any metric tree $T$ with $n$ edges of infinite length. Then there are $n$ resulting boundary lines (think of the boundary of an $\epsilon$-thickening of $T$ and let $\epsilon\to 0$) and one can attach $n$ euclidean half-planes to these boundary lines by isometries along their boundaries. The resulting Riemann surface is simply connected, and parabolic, and hence $\mathbb{C}$, and is equipped with an hpd with order $n+2$ pole at infinity. The  metric spine is $T$, and the metric residue at the pole can be read off from the lengths of edges in $T$ (see Figure 7).

\subsubsection*{Multiple-poled} Consider $\mathbb{C}$ obtained by gluing half-planes to a metric tree $T$ as above. The following local modification introduces another pole: Take a subinterval of one of the edges of $T$ and slit it open, introduce $n^\prime$ vertices on the resulting boundary circle, and attach $n^\prime$ semi-infinite edges from those vertices. Along each of the $n^\prime$ resulting new boundary lines, we attach half-planes as before. Topologically, one has just attached a punctured disk to $\mathbb{C}$ after the slit, so the resulting surface (after adding the puncture) is still $\hat{\mathbb{C}}$, but this now has a half-plane differential with a  new pole of order $n^\prime +2$.  One can also vary the residues at the poles by controlling edge lengths (see Figure 8).

\begin{figure}[h]
  \centering
  \includegraphics[scale=0.65]{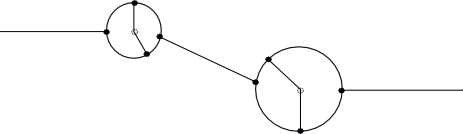}\\
  \caption{An hpd on $\hat{\mathbb{C}}$: The zeroes are the dark vertices, and the poles of order $4$ are at the lighter vertices and infinity. Varying the lengths of various edges changes the residues and leading order terms. }
  \end{figure}

\subsection{Interval-exchange surfaces}

Introduce a finite-length horizontal slit on $\mathbb{C}$ and glue the resulting two sides of the interval by an interval exchange (see Figure 9). The resulting surface (punctured at infinity) can be of any genus, by prescribing the combinatorics of the gluing appropriately. By adding the two semi-infinite horizontal intervals on either side of the slit, one can easily see that this is a half-plane surface with two half-planes (above and below the horizontal line). The half-plane differential has a single order-$4$ pole at infinity. \\

\begin{figure}[h!]
  \centering
  \includegraphics[scale=0.8]{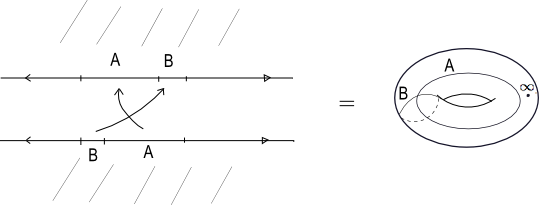}\\
  \caption{The half-plane surface obtained by the interval exchange shown on the left gives a punctured torus.  This is not the generic case, as the resulting hpd has a zero of order two.}
\end{figure}

One consequence of Theorem \ref{thm:main} is that one gets \textit{all} once-punctured Riemann surfaces this way. The following is a simple dimension count that indicates that this is possible:\\

We place a vertex at infinity to get a cell decomposition of a \textit{closed} surface of genus $g$. We have
\begin{equation}\label{eq:euler}
v-e+f = 2-2g
\end{equation}
where $v$, $e$ and $f$ are the number of vertices, edges and vertices of the resulting decompistion.\\

Note that there are two faces (the two half-planes), so $f=2$.  Moreover, in the generic case, all vertices are trivalent except one (the vertex at infinity, which has valence $2$). Hence
\begin{equation*}
2e = 3(v-1) +2
\end{equation*}
Using these facts in  (\ref{eq:euler}) we get 
\begin{equation*}
e= 6g-1
\end{equation*}

To count the resulting number of parameters, note that there are two edges corresponding to the two semi-infinite edges, and the conformal structure of the half-plane surface does not change if one scales all finite lengths by a positive real. Hence the total dimension of the set of parameters is $(6g-4)$, which is the dimension of the moduli space $\mathcal{M}_{g,1}$.

\subsection{Single-poled hpd's on surfaces}

Take a single-poled hpd on $\hat{\mathbb{C}}$ and introduce a slit on one of the edges of the metric spine, and glue the resulting two sides by an interval exchange. By suitable choice of combinatorics of gluing, this gives higher genus half-plane surfaces. (See Figure 9)

\begin{figure}[h!]
  \centering
  \includegraphics[scale=0.9]{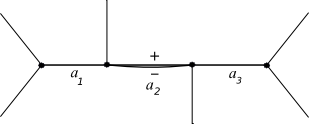}\\
  \caption{Slitting the metric spine of an hpd on $\hat{\mathbb{C}}$ along an edge and then gluing the resulting sides by an interval exchange produces a single-poled hpd on higher genus surfaces. }
  \end{figure}

\subsection{A low complexity example:}

The following lemma deals with the exceptional case in Theorem \ref{thm:main}. 

\begin{lem} Any meromorphic quadratic differential on $\hat{\mathbb{C}}$  with a single pole $p$ of order 4 has residue $0$ at $p$.
\end{lem}
\begin{proof}
Let $q$ be such a meromorphic quadratic differential, so it is 
\begin{equation*}
q  = (\frac{c_4}{z^4} + \frac{c_3}{z^3} + \cdots)dz^2
\end{equation*}
in the usual coordinates in an affine chart. 
However in $\hat{\mathbb{C}}$ we also have the quadratic differential 
\begin{equation*}
\psi = \frac{c_4}{z^4}dz^2
\end{equation*}
The quadratic differential  $q-\psi$ then has a single pole of order less than or equal to $3$. There is no such non-zero quadratic differential on $\hat{\mathbb{C}}$, and hence $q =\psi$, and has residue $0$.
\end{proof}

\section{Outline of the proof}

We illustrate the proof of Theorem \ref{thm:main}  in the case of a single puncture. This easily generalizes to the case of multiple poles - in \S12 we shall provide a summary. \\

Throughout, we shall fix a Riemann surface $\Sigma$ of genus $g$, with a marked point $p$ with a disk neighborhood $U$, and an integer $n\geq 4$ and $a,c\in \mathbb{R}^{+}$. Our goal is to show there exists a  conformal homeomorphism 
\begin{equation*}
g:\Sigma\setminus p \to \Sigma_{n,a}
\end{equation*}
where $\Sigma_{n,a}$ is a half-plane surface such that the half-plane differential has one pole of order $n$ and residue $a$. Furthermore, via the uniformizing chart 
\begin{equation*}
\phi:U\to \mathbb{D}
\end{equation*}
taking $p$ to $0$ the pullback quadratic differential on $\mathbb{D}$ has a pole at $0$ with leading order term $c$.\\

\medskip

Briefly, the argument consists of producing a sequence of half-plane surfaces (\textit{Steps 1} and \textit{2}) that converge to a surface conformally equivalent to $\Sigma \setminus p$ (\textit{Step 3}) and metrically a half-plane surface (\textit{Steps 4} and \textit{5}). The bulk of the proof lies in proving the latter convergence, after first showing that one has sufficient geometric control on the surfaces along the sequence.\\

\textbf{Step 1.} \textit{(Quadrupling)} We define a suitable compact exhaustion $\{\Sigma_i\}$  of $\Sigma\setminus p$, and by a two-step conformal doubling procedure along boundary arcs  we define a corresponding sequence of compact Riemann surfaces $\hat{\Sigma_i}$. An application of the Jenkins-Strebel  theorem then produces certain holomorphic quadratic differentials on these surfaces which on passing back to $\Sigma_i$ again by the involutions gives singular flat structures with ``polygonal" boundary.\\

\textbf{Step 2.} \textit{(Prescribing boundary lengths)}  We complete each of these singular flat surfaces with polygonal boundary to a half-plane surface $\Sigma_i^\prime$  by gluing in an appropriate planar end. We first show that by choosing the arcs in the first doubling step appropriately, one can ensure that the sequence of planar ends one needs are truncations at height $H_i\to \infty$ of a fixed planar end $\mathcal{P}$. Here $H_i$ is a sequence of real numbers diverging at a prescribed rate. This is the geometric control crucial for the convergence in \textit{Step 4}.\\

\textbf{Step 3.}\textit{ (Conformal limit)} Applying a  quasiconformal extension result  we show that these half-plane surfaces $\Sigma_i^\prime$  have $\Sigma\setminus p$ as a conformal limit.\\

\textbf{Step 4.}\textit{ (The quadratic differentials converge)} We now show that the half-plane differentials corresponding to $\Sigma_i^\prime$ satisfy a convergence criterion (see Appendix A) and hence after passing to a subsequence they converge to a meromorphic quadratic differential with the right order and residue. \\

\textbf{Step 5.} \textit{(A limiting half-plane surface)} We show that the limiting quadratic differential  in \textit{Step 4} is in fact a half-plane differential, that is, the sequence of half-plane surfaces limits to a half-plane surface $\Sigma_{n,a}$. By \textit{Step 3}, this surface is conformally $\Sigma \setminus p$, as required.\\

\textbf{Step 6.}\textit{ (The leading order coefficient)} With a final analytical lemma we show that an additional control on the sequence $H_i\to \infty$ in \textit{Step $2$} ensures that the limiting half-plane differential has leading order term $c$ on $U$, as required.

\section{Step 1: A quadrupling procedure}

 \subsection*{The compact exhaustion}
 
Consider the  neighborhood $U$ of $p\in \Sigma$ with the conformal chart $\phi:U\to \mathbb{D}$ such that $\phi(p) = 0$. Let $B(r) \subset \mathbb{D}$ denote the open disk of radius $r$ centered at $0$, and let $U(r) \subset U$ denote the inverse image $\phi^{-1}(B(r))$.\\

Define $\Sigma_i$ to be Riemann surface $\Sigma \setminus U(2^{-i})$. For convenience we shall denote $U(2^{-i})$ by $U_i$. Note that this a compact Riemann surface with boundary $C_i = \partial \Sigma_i = \partial  \overline{U(2^{-i})}$.\\

\begin{figure}[h!]
  \centering
  \includegraphics[scale=0.85]{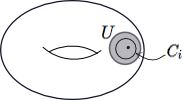}\\
  \caption{$C_i$ is the inverse image of a circle of radius $2^{-i}$ under the conformal chart $\phi:U\to \mathbb{D}$. }
  \end{figure}

The subsurfaces $\{\Sigma_i\}_{i=1}^{\infty}$ form a compact exhaustion. In particular, we note the following: \\
(1) $\Sigma_i\subset \Sigma_{i+1}$ for each $i\geq 1$.\\
(2) $\bigcup\limits_{i=1}^\infty \Sigma_i= \Sigma \setminus p$.\\
(3) $U_i = \Sigma \setminus \Sigma_i$ is a topological disk containing $p$, and\\
(4) $\Sigma_{i}\setminus \Sigma_{1}$ is a topological annulus of modulus $A\cdot i$ for some constant $A>0$.

\subsection*{Conformal doubling}

For each subsurface in the compact exhaustion constructed in the previous section, we shall now define a two-step doubling across a collection of $n$ arcs to get a sequence of compact surfaces $\widehat{\Sigma_i}$.\\

For each $i\geq 1$  consider the homeomorphism $\phi_i:\overline{U_i} \to \overline{B(2^{-i})}$ that is the restriction of the conformal chart $\phi$. Choose a collection of $n$ arcs on the boundary $\partial \overline{B(2^{-i})}$ and pull it back via $h_i^{-1}$ to a collection of arcs $a_1,a_2,\ldots a_n$ on $C_i = \partial \Sigma_i$. Note that the complement of these arcs on $C_i$ is another collection of $n$ arcs which we denote by $b_1,b_2,\ldots b_n$. In \textit{Step 2} (following section) we shall specify more about the choice of these arcs.
\\

Consider now two copies of the surface $\Sigma_i$ with the collections of $a$ and $b$ arcs on its boundary. Passing to the unit disc $\mathbb{D}$ via the conformal chart $h$, glue the $b$ arcs together via the anti-conformal map $z\mapsto 2^{-2i}/\bar{z}$ (this preserves the circle of radius $2^{-i}$). We get a doubled surface $\Sigma^d_i$ with a conformal structure, which has $n$ slits corresponding to the $a$ arcs on the boundary of each half (which remain unglued in our doubling). \\

\begin{figure}[h]
  \centering
  \includegraphics[scale=0.55]{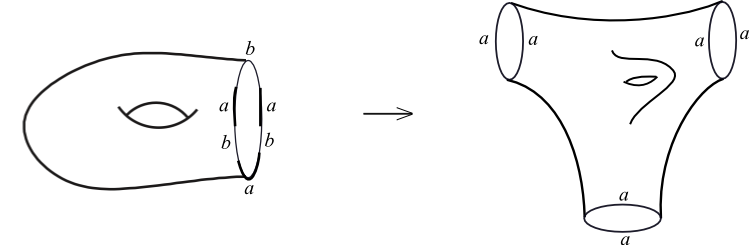}\\
  \caption{We first double $\Sigma_i$ (shown on the left) along the $b$-arcs on the boundary to get $\Sigma^d_i$ (shown on the right). }
  \end{figure}

Next, we take two copies of this resulting surface $\Sigma_d^i$ and glue them along these slits to get a \textit{closed} Riemann surface $\widehat{\Sigma_i}$. This time the conformal gluing is via a suitable restriction of the hyperelliptic involution of a genus $n$ surface branced over $n$ equatorial slits on $\hat{\mathbb{C}}$ (the restriction is to a collar neighhborhood on one side of the equator). \\

Note that the glued pairs of $a$-slits form a collection of $n$ nontrivial homotopy classes of curves $[\gamma_1],[\gamma_2],\ldots [\gamma_n]$ on the surface  $\widehat{\Sigma_i}$.\\

  \begin{figure}[h]
  \centering
  \includegraphics[scale=0.5]{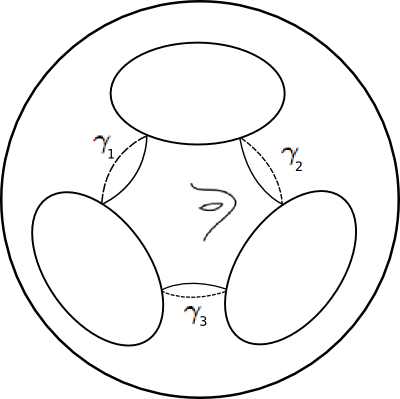}\\
  \caption{In the second step one glues two copies of $\Sigma_d^i$ (see Figure 12) to get the closed, ``quadrupled" surface $\widehat{\Sigma_i}$. (Not all the handles are shown in the figure.) }
  \end{figure}

We also have a pair of anticonformal involutions  $j^1_i$ and $j^2_i$ , where $j_i^1:  \Sigma^d_i \to  \Sigma^d_i$ is the deck translation of the branched double covering $\pi^1_i:  \Sigma^d_i \to  \Sigma_i$, and $j^2_i:\widehat{\Sigma_i}\to \widehat{\Sigma_i}$ that similarly commutes with the branched double covering $\pi^2_i:  \widehat{\Sigma_i} \to  \Sigma_d^i$.

\subsection*{Rectangular surfaces}

By the theorem of Jenkins-Strebel,  on each surface $\widehat{\Sigma_i}$ we have a holomorphic quadratic differential $\hat{q_i}$ which induces a singular-flat metric comprising $n$ euclidean cylinders of circumference $(2H_i, 2H_i,\ldots 2H_i+2a)$ with core curves $[\gamma_1],[\gamma_2],\ldots [\gamma_n]$, where 
\begin{equation}\label{eq:hi}
H_i=\left(H_0\cdot 2^i\right)^{n/2}
\end{equation}
for a choice of a $H_0>0$ that shall be eventually made in Proposition \ref{prop:deriv}.\\

(The reason for choosing $H_i$ to be of the above form shall be clarified by Lemma \ref{lem:uhi}).\\

We now show that this passes down to a singular flat metric on $\Sigma_i$ with a ``rectangular" structure when we quotient back by the anticonformal involutions $j^1_i$ and $j^2_i$.

\begin{lem}\label{lem:diam}
Let $X$ be a Riemann surface with a holomorphic quadratic differential $q$, and let $j:X\to X$ be an anti-conformal involution that is pointwise identity on a connected analytic arc $\gamma$. Then $\gamma$ is either completely horizontal or completely vertical in the quadratic differential metric.
\end{lem}

\begin{proof}

Let $p\in \gamma$, and let $v\in T_pX$ be the tangent vector to $\gamma$ at $p$. Since $\gamma$ is fixed pointwise by $j$, the induced map $j_\ast:T_pX\to T_pX$ satisfies $j_\ast(v)=v$.\\

Since $j$ is anticonformal, we have that the pullback quadratic differential $j^\ast q$ satisfies:
\begin{equation}\label{eq:eq1}
j^\ast q (v,v) = \overline{ q (v,v) }
\end{equation}
where $\bar{\alpha}$ denotes the complex-conjugate of a complex number $\alpha$.\\

On the other hand,  by definition of the pullback we have:
\begin{equation}\label{eq:eq2}
j^\ast q (v,v) = q(j_\ast v,j_\ast v) = q(v,v)
\end{equation}
where we used the fact that $j_\ast$ fixes $v$.\\

By (\ref{eq:eq1}) and (\ref{eq:eq2}) we have that $\overline{q(v,v)}=q(v,v)$ and hence $q(v,v)\in \mathbb{R}$.\\

Consider the local coordinate around any point  $p\in \gamma$ in which the quadratic differential $q$ is $dz^2$. By the previous observation and the fact that $\gamma$ is analytic,  $\gamma$ is locally either a horizontal or vertical segment around the image of $p$. Since $\gamma$ is connected, this is true of the entire arc.
\end{proof}

We apply this lemma to the involutions $j^1_i$ and $j^2_i$. First, the anticonformal involution $j^2_i:\widehat{\Sigma_i}\to \widehat{\Sigma_i}$ fixes the $n$ $a$-slits and hence they are either completely horizontal or vertical. Since they are homotopic to the core curves of the cylinders in the  $\hat{q_i}$-metric, they must in fact be completely horizontal. Next, the anticonformal involution $j_i^1:  \Sigma^d_i \to  \Sigma^d_i$ fixes the $b$-arcs on $\Sigma^d_i$. Since these arcs embed in $\widehat{\Sigma_i}$ as transverse arcs across the cylinders in the $\hat{q_i}$-metric, they must be completely vertical.\\

Hence the holomorphic quadratic differential $\hat{q_i}$ passes down to a holomorphic quadratic differential $q_i$ on the bordered surface $\Sigma_i$. The $n$  $a$-arcs on  $\partial \Sigma_i$ become horizontal segments, and the remaining $n$ $b$-arcs on $\partial \Sigma_i$ are vertical segments. These form a polygonal boundary. The $n$ euclidean cylinders on $\widehat{\Sigma_i}$  descend to a cyclically-ordered collection of $n$ euclidean rectangles on $\Sigma_i$ glued along the critical graph $\mathcal{G}_i$. In the  $q_i$-metric on  $\Sigma_i$, the horizontal width of these rectangles are $(H_i, H_i\ldots H_i+a)$ in a cyclic order (see Figure 14).

\begin{figure}[h]
  \centering
  \includegraphics[scale=0.55]{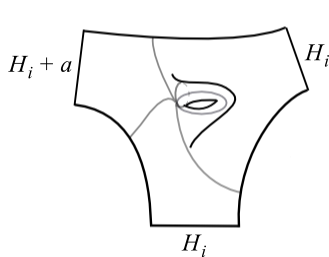}\\
  \caption{The rectangular surface at the end of \textit{Step 1}.}
  \end{figure}

\section{Step 2: Prescribing lengths}
In \textit{Step 1}, the Jenkins-Strebel theorem allows us to prescribe the circumferences of the resulting metric cylinders on the ``quadrupled" surface $\widehat{\Sigma_i}$, but not their lengths. On quotienting back by the two involutions, this results in control on the lengths of the ``horizontal"  sides of the polygonal boundary of the resulting ``rectangular" surface (see figure above). We show here, using a continuity method, that choosing the arcs carefully in the conformal doubling step ensures that the extremal lengths of the curves obtained from the doubled arcs on $\widehat{\Sigma_i}$ are appropriate values (Lemma \ref{lem:arcs}) such that the vertical edge-lengths are also prescribed (Lemma \ref{lem:cyllen}).\\

We shall use the following:  

\begin{lem}[A topological lemma]\label{lem:deglem}
Let $\phi:\mathbb{R}^n_{>0} \to \mathbb{R}^n_{>0}$ be a proper, continuous map mapping 
\begin{equation*}
(x_1,x_2,\ldots ,x_n) \mapsto (y_1,y_2,\ldots y_n)
\end{equation*}
Suppose there exists functions $\eta_1,\eta_2$ such that:\\
(1)  $x_i>A \implies y_i>\eta_1(A)$, and\\
(2)  $x_i<\epsilon \implies y_i< \eta_2(\epsilon)$,\\
 for each $1\leq i\leq n$, where $\eta_1(A) \to \infty$ as $A\to \infty$, and $\eta_2(\epsilon)\to 0$ as $\epsilon\to 0$.\\
Then $\phi$ is surjective.\\
 (Note: here $\mathbb{R}_{>0} = \mathbb{R}^{+}$ denotes the positive real numbers.)
\end{lem}
\begin{proof}[Sketch of a proof]
The proof is a standard topological degree argument. Note that (1) and (2) are equivalent to the requirement that $\phi$, in addition to being proper, has a coordinate-wise control: for a sequence of points $(x_1^j,x_2^j,\ldots x_n^j)$ and $\phi$-images $(y_1^j,y_2^j,\ldots y_n^j)$ (where $j\in\mathbb{N}$) we have that fixing $1\leq i\leq n$, $x_i^j\to\infty \implies y_i^j\to \infty$ and $x_i^j\to 0 \implies y_i^j\to 0$ \textit{uniformly} (independent of the rest of the coordinates). This implies that $\phi$ has degree $1$ at infinity, and hence $\phi$ is surjective.
\end{proof}

\subsection*{Arcs and extremal lengths}
Consider a Riemann surface $\Sigma$ with one boundary component which we identify with $S^1$, and with $n$ sub-intervals $I_1 ,I_2,\ldots I_n$ of equal length. Each $I_i$ is further divided into two sub-arcs $\{a_i,b_i\}$ in clockwise order (see figure). As in the construction in \textit{Step 1}, consider the two step doubling that leads to a a closed Riemann surface: first double along the $b$ arcs on $\partial \Sigma$ to get a Riemann surface $\Sigma_d$ with slits corresponding to the $a$- arcs, and next, glue two copies of $\Sigma_d$ along these slits to get a closed surface $\widehat{\Sigma}$.\\

 We shall work with the doubled surface $\Sigma_d$, and homotopy classes of curves $\gamma_i$ that enclose each of the slits $a_i$, for $1\leq i\leq n$. Let their extremal lengths be $\lambda_1,\lambda_2,\ldots \lambda_n$ - note that these values are exactly double of the extremal lengths of the corresponding closed curves one gets on $\widehat{\Sigma}$. We shall first show that by prescribing the $2$-arc decomposition of each interval $I_i$ appropriately, we can obtain any $n$-tuple of extremal lengths.\\
 
Consider the subinterval $I_i =  a_i\cup b _i$. We denote the (angular) length of a subarc $\tau$ on $\partial \Sigma \equiv S^1$ by $l(\tau)$. Note that $l(I_i) = \frac{2\pi}{n}$ for each $i$. For convenience, we shall fix a conformal metric $\rho$ on $\Sigma$ that gives a length $2\pi$ to the boundary circle $\partial \Sigma$, and the above lengths of arcs shall be those induced by this metric. Denote the ratio of lengths $r_i = \frac{l(a_i)}{l(b_i)}$. Also, notice that each $a_i$ arc has two adjacent arcs $b_{i-1}$ and $b_i$ on either side (where $i-1$ is taken to be $n$ if $i=1$).

\begin{figure}[h]
  \centering
  \includegraphics[scale=0.6]{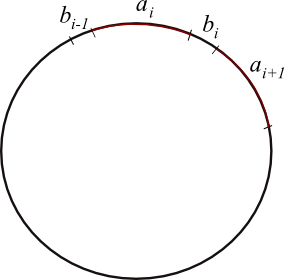}\\
  \caption{The $a$- and $b$-arcs on the boundary component of $\partial \Sigma$. The closed curve $\gamma_i$ goes around the $a_i$ slit on the surface doubled across the $b$-arcs, and has extremal length $\lambda_i$. }
\end{figure}

\begin{lem}\label{lem:arcs} The map $\phi:\mathbb{R}^n_{\geq 0}\to \mathbb{R}^n_{\geq 0}$ that assigns to a tuple $(r_1,r_2,\ldots r_n)$ of ratios of interval lengths, the corresponding extremal lengths $(\lambda_1,\lambda_2,\ldots,\lambda_n)$ is surjective.
\end{lem}
\begin{proof} The map $\phi$ is continuous since the moduli of the doubled surfaces depend continuously on the lengths of the slits (even if some slits degenerate to punctures). The lemma shall follow once we show that the following properties hold.\\

\textit{Notation.} In what follows we shall consider an arbitrary $(r_i)_{1\leq i\leq n}$ in $\mathbb{R}^n_{\geq 0}$ and its $\phi$-image $(\lambda_i)_{1\leq i\leq n}$. In (2) and (3) below, we consider a sequence $\{(r_i)^j\}$ of such $n$-tuples and their $\phi$-images $\{(\lambda_i)^j\}$, where the index $j$ runs from $1\leq j<\infty$ and shall be appended to any of the geometric quantities varying with $j$.\\

(1) $\lambda_i=0 \iff r_i=0$.\\
(2) $r_i^j\to \infty \implies \lambda_i^j\to \infty$. The divergence is uniform, that is $r_i>c \implies \lambda_i>\eta_1(c)$.\\
(3) $r_i<c \implies \lambda_i<\eta_2(c)$.\\
(Here $\eta_1,\eta_2:\mathbb{R}_{\geq 0} \to \mathbb{R}_{\geq 0}$ are increasing functions.)\\

Note that (1) and (2) imply that the map $\phi$ is proper, and from the uniform estimates in (2) and (3), the surjectivity of $\phi$ follows from Lemma \ref{lem:deglem}. \\

\textit{Property (1)}: The backward implication holds since $r_i=0\iff l(a_i)=0$ and hence the corresponding $a$-slit has degenerated to a puncture on $\Sigma_d$, and the extremal length $\lambda_i$ of a loop enclosing a puncture is $0$. For the other implication, observe if $l(a_i) \neq 0$ then for our choice of conformal metric $\rho$ we shall have $l_\rho(a_i) >\eta >0$, and a lower bound of $2\eta$ of the length of the curve around the $a_i$-slit. The analytic definition of extremal length:
\begin{equation*}
\lambda_i = \sup\limits_\rho \frac{l_\rho(a_i)^2}{A(\rho)}
\end{equation*}
then shows that there is a positive lower bound on the extremal length $\lambda_i$.\\

\textit{Property (2)}: Note that by definition, $r_i^j\to \infty \implies l(b_i^j)\to 0$. By the geometric definition of extremal length, 
\begin{equation}\label{eq:mod}
\lambda_i^j =\inf \frac{1}{mod(\mathcal{A})}
\end{equation}
where $\mathcal{A} \subset \Sigma_d$ is an embedded annulus  with core-curve enclosing the $a_i^j$-slit.\\
 It is well-known that there is a bound $B$ on the largest modulus annulus $\hat{\mathcal{A}}$ that can be embedded  in $\mathbb{C}$ such that the bounded complementary component contains the interval $[0, l(a_i^j)]$ and the other component contains the point $ l(a_i^j) + l(b_i^j) \in \mathbb{R}$. Moreover, $B\to 0$ as $l(b_i)\to 0$. In our case, the $a_i^j$-slit on $\Sigma_d$ is such an interval in an appropriate conformal chart, and the endpoint of the adjacent b-slit is the other real point. However since the annulus $\mathcal{A}$ is now constrained to be embedded in the surface $\Sigma_d$,  the modulus of $\mathcal{A}$ is less than that of $\hat{\mathcal{A}}$ (in the planar case). One can also see this by considering the annular cover associated with the closed curve where all the slits lie. From the above discussion, this proves that $\lambda_i^j\to \infty$ by (\ref{eq:mod}). \\

\textit{Property (2) continued}: 
The uniform divergence follows by quantifying the estimates in the argument above: the bound $B$ of the largest modulus of an annulus in $\mathbb{C}$ separating the interval $[0,a]$ from $a+b\in \mathbb{R}$ is in fact a strictly increasing function of $b/a$. That is,  $B = \eta(b/a)$ where $\eta:\mathbb{R}_{\geq 0} \to \mathbb{R}_{\geq 0}$ is a continuous function such that $\eta(0) = 0$.  The above argument then shows that 
\begin{equation*}
r_i>c \implies \frac{l(b_i)}{l(a_i)} < 1/c \implies B < \eta(1/c) \implies \lambda_i > \eta(1/c)^{-1}
\end{equation*}
where we have used (\ref{eq:mod}) for the last inequality. Hence the uniform estimate holds with the function $\eta_1(x) := \eta(1/x)^{-1}$.\\

\textit{Property (3)}: This follows from closely examining the argument of the backward implication in \textit{Property (1)}: if $r_i<c$ then by definition $l(a_i)< \frac{2\pi c}{n(1+c)}$ and hence for sufficiently small $c$ one can embed an annulus of inner radius $\frac{2\pi c}{n(1+c)}$ and outer radius $R$ (that depends only on $\rho$ and $\Sigma$) on the doubled surface $\Sigma_d$. The modulus of this annulus is $M(c)$ which tends to $\infty$ as $c\to 0$. By the geometric definition of extremal length (\ref{eq:mod}), $\lambda_i$ is less than $1/M(c)$.

\end{proof}

\textit{Remark.} A similar setup involving slits along subintervals of the real line in $\hat{\mathbb{C}}$ was considered in \cite{Penner}.

\subsection*{Cylinder lengths}

As in \S6, an application of the Jenkins-Strebel theorem to the homotopy classes of curves (corresponding to the doubled $a$-slits) on the ``quadupled" surface $\widehat{\Sigma}$ produces a quadratic differential metric with a decomposition into $n$ metric cylinders $C_1,C_2,\ldots C_n$ with these as the core curves. By choosing the $a$-arcs on $\partial \Sigma$ appropriately, by Lemma \ref{lem:arcs} these curves can have assume any $n$-tuple of extremal lengths.  We now show that by prescribing these extremal lengths correctly, one can assume any $n$-tuple of cylinder lengths $\{l_i\}_{1\leq i\leq n}$.  Note that the Jenkins-Strebel theorem already allows one to prescribe arbitrary circumferences.

\begin{lem}\label{lem:cyllen}
Suppose one fixes each cylinder circumference to be $H$. Then the map $\psi:\mathbb{R}^n_{>0}\to \mathbb{R}^n_{> 0}$ that assigns to a tuple $(\lambda_1,\lambda_2,\ldots \lambda_n)$ of extremal lengths of the core curves, the reciprocals of the cylinder lengths $(1/l_1,1/l_2,\ldots,1/l_n)$, is surjective.
\end{lem}
\begin{proof}
The surjectivity of $\psi$ shall follow from Lemma \ref{lem:deglem} once we establish the following properties:\\

(1) $\lambda_i> C \implies  1/l_i > \eta_1(C)$.\\
(2) For sufficiently small $c$, $\lambda_i < c \implies 1/l_i < \eta_2(c)$.\\

for some increasing functions $\eta_1, \eta_2:\mathbb{R}_{\geq 0} \to \mathbb{R}_{\geq 0}$.\\

\textit{Property (1)}: By (\ref{eq:mod}) we have:
\begin{equation*}
\lambda_i\geq C \implies mod(\mathcal{A}) \leq 1/C
\end{equation*}
where $\mathcal{A}$ is any embedded annulus in $\Sigma_d$ with core curve $\gamma_i$. This implies that $l_i \leq H/C$ since otherwise one can embed a flat cylinder of modulus greater than $1/C$ with core curve $\gamma_i$. Hence $\eta_1(x) = x/H$ works.\\

\textit{Property (2)}: We shall prove the converse statement: If $l_i<B$, then $\lambda_i>C>0$ for some $C$ that depends on $B$. For this, we shall construct a curve $\beta_i$ that intersects $\gamma_i$ at most twice, such that the extremal length
 \begin{equation}\label{eq:bi}
 Ext(\beta_i) < D
 \end{equation}
 for some $D$ (depending on $B$). The lower bound $C$ for $\lambda_i =Ext(\gamma_i)$ now follows from the well-known inequality (see \cite{Min1}):
\begin{equation*}
Ext(\gamma_i)Ext(\beta_i)\geq i(\gamma_i, \beta_i)^2 
\end{equation*}

To show the bound (\ref{eq:bi}) we shall use the geometric definition of extremal length (\ref{eq:mod}): namely, we shall construct an annulus of definite modulus with core curve $\beta_i$.\\

The construction of $\beta_i$ is geometric, and falls in two cases. Consider the cylinder  $C_i$ corresponding to $\gamma_i$ on the quadrupled surface $\Sigma_d$, of circumference $H$, and length $l_i<B$ (by our current assumption).  Recall $C_i$ has a bilateral symmetry that comes from the doubling. Each of its two boundary components is adjacent to the other cylinders $C_2,\ldots C_n$, and at least one of them, say $C_k$, shares a boundary segment with $C_i$ of definite length, that is: 
\begin{equation*}
l(\partial C_k \cap \partial C_i)> \frac{H}{n}
\end{equation*}
Let the length of the cylinder $C_k$ be $l_k$. The two cases are:\\

(I) $l_k \leq B$:  In this case we construct the curve $\beta_i$ intersecting $\gamma_i$ once, as shown in Figure 16.\\

  \begin{figure}[h]
  \centering
  \includegraphics[scale=0.65]{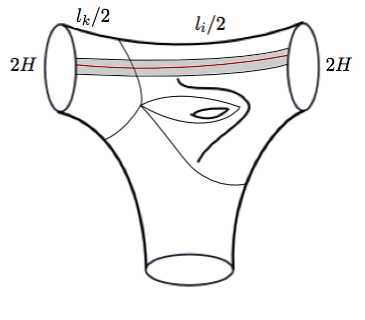}\\
  \caption{In Case I $\beta_i$ consists of an arc across $C_i$ and an arc across $C_k$. This figure shows half of the surface. By the two-fold symmetry from the ``doubling", these arcs join up to form a closed curve.  }
  \end{figure}

(II) $l_k>B$: In this case we construct $\beta_i$ intersecting $\gamma_i$ twice, as shown in Figure 17.\\

  \begin{figure}[h]
  \centering
  \includegraphics[scale=0.65]{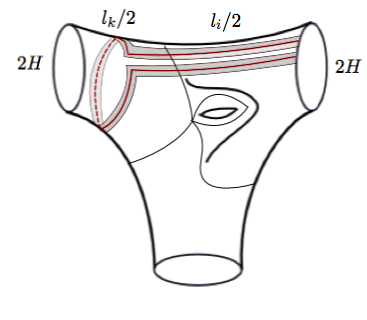}\\
  \caption{In Case II $\beta_i$ comprises two arcs across $C_i$, separated by a definite distance, that extend a bit into $C_k$,  together with a loops around $C_k$ connecting the two pairs of endpoints.}
  \end{figure}

In both cases, the curves are of bounded length and admit an embedded ``collar" neighborhood of definite width (these dimensions depend only on $B$ and $H$) and hence satisfy (\ref{eq:bi}) for some $D$, by the geometric definition of extremal length. The function $\eta_2$ is implicit from the construction. \end{proof}

\textit{Remark.} The above argument also works if the cylinder circumferences are fixed $n$-tuple $(c_1,c_2,\ldots c_n)$, by replacing $H$ by the maximum or minimum value of the $c_i$-s, in the proof, as appropriate.

\subsection*{Half-plane surfaces $\Sigma_i^\prime$}

By Lemmas \ref{lem:cyllen} and \ref{lem:arcs} one can now choose the $a$- and $b$-arcs on $\partial \Sigma_i$ such that the cylinder lengths of the Jenkins-Strebel differential on the quadrupled surface $\widehat{\Sigma_i}$ are all $2H_i$. (Recall $H_i = \left(H_0\cdot 2^i\right)^{n/2}$ as in (\ref{eq:hi}). See also the remark following Lemma \ref{lem:cyllen}.)\\

Hence on each singular flat surface $(\Sigma_i,q_i)$, one now has $n$ euclidean rectangles $R_1,\ldots R_n$ glued along the metric spine $\mathcal{G}_i$ in that cyclic order, such that the resulting polygonal boundary has all the side-lengths $H_i$, except one horizontal side of length $H_i+a$ (see Figure 18 - here $a$ is the desired residue at the pole).\\
  
We construct a half-plane surface $\Sigma_i^\prime$ by gluing in a planar end $\mathcal{P}_{H_i}$ (see also Definition \ref{defn:trunc}) which has a polygonal boundary isometric to $\partial \Sigma_i$. From our choice of lengths of the polygonal boundary, the metric residue of $\Sigma_i^\prime$ is equal to $a$. For each $i$ these planar ends are truncations at height $H_i\to \infty$ of a fixed planar end $\mathcal{P}$ of residue $a$.

\subsection*{Choice of $H_i$} 

Recall from \S3.2 that there is a conformal map from $\mathcal{P}_{H_i} \subset \Sigma^\prime_i$ to a neighborhood $U_{H_i}$ of $0\in\mathbb{C}$ that is an isometry in the $\phi$-metric as in (\ref{eq:stan}).\\

We now observe that by Lemma \ref{lem:uh} the choice of $H_i$ in (\ref{eq:hi}) yields the following:

\begin{lem}\label{lem:uhi}
$D_1^\prime2^{-i}\leq dist(0,\partial U_{H_i})  \leq D_2^\prime2^{-i}$, where $D_1^\prime, D_2^\prime>0$ are constants independent of $i$ (they depend only on the choice of $H_0$).
\end{lem} 

\textit{Remark.} This choice implies the modulus of the annulus $U\setminus U_i$ in the compact exhaustion is comparable (upto a bounded multiplicative factor) to that of the annulus $\mathcal{P}_{H_0}\setminus \mathcal{P}_{H_i}$ on the half-plane surface $\Sigma_i^\prime$. This is the geometric control crucial for extracting a convergent subsequence in \textit{Step 4}.

\begin{figure}[h]
  \centering
  \includegraphics[scale=0.7]{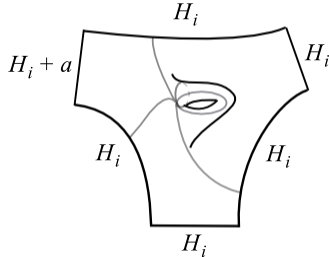}\\
  \caption{\textit{Step 2} ensures that one gets a ``rectangular" surface $\Sigma_i$ with horizontal and vertical edges as shown, that can be extended to a half-plane surface $\Sigma_i^\prime$ by gluing in a planar end. }
  \end{figure}

\section{Step 3: A conformal limit}

Here we show that the sequence $\{\Sigma_i^\prime\}_{i\geq 1}$ of half-plane surfaces constructed in the previous section has $\Sigma\setminus p$ as a conformal limit:

\begin{lem}[Conformal limit]\label{lem:approx1}
For all sufficiently large $i$ there exist $(1+\epsilon_i)$-quasiconformal homeomorphisms
\begin{equation}\label{eq:fi}
f_i:\Sigma_i^\prime \to \Sigma\setminus p
\end{equation}
where $\epsilon_i\to 0$ as $i\to \infty$.
\end{lem}

The intuition behind the proof is that since $\Sigma_i^\prime$ is obtained by excising-and-regluing disks from $\Sigma$ that get smaller as $i\to\infty$, for large enough $i$ the conformal structure is not too different, and one can construct an almost-conformal map as above.

\subsection*{Quasiconformal extensions}
The following lemma is a slight strengthening of the quasiconformal extension lemma proved in \cite{Gup1}.\\

Throughout, $\mathbb{D}$ shall denote a unit disk of radius $1$ and $B(r)$ shall denote an open ball of radius $r$, centered at $0\in \mathbb{C}$.

\begin{lem}\label{lem:qclemnew}
For any $\epsilon>0$ sufficiently small, and $0\leq r\leq \epsilon$, a map
\begin{center}
$f:\mathbb{D}\setminus B(r)\to \mathbb{D}$
\end{center}
that\\
(1) preserves the boundary and is a homeomorphism onto its image,\\
(2) is $(1+\epsilon)$-quasiconformal on $\mathbb{D}\setminus B(r)$ \\
extends to a $(1 + C\epsilon)$-quasisymmetric map on the boundary, where $C>0$ is a universal constant.
\end{lem}

\begin{proof}[Sketch of the proof]
In \cite{Gup1} (see Appendix A of that paper) we proved this when the map $f$ was a quasiconformal homeomorphism of the \textit{entire} disk, though it had the control on distortion only on the annulus $A = \mathbb{D}\setminus B(r)$ as in $(2)$ above. However, all that was required was the following estimate on the image of the ball $B(r)$:
\begin{equation*}
diam(f(B(r))) < C_1\epsilon
\end{equation*}
for some universal constant $C_1>0$.\\
Here, this can be replaced by the following fact:
\begin{equation}\label{eq:pf1}
d = diam(\mathbb{D}\setminus f(A)) < C_1\epsilon
\end{equation}
which follows from the modulus-estimates
\begin{equation}\label{eq:pf2}
\frac{1}{1+\epsilon} \leq \frac{mod(f(A))}{mod(A)}
\end{equation}
\begin{equation}\label{eq:pf3}
mod(A) = \frac{1}{2\pi}\ln{\frac{1}{r}} \leq \frac{1}{2\pi}\ln{\frac{1}{\epsilon}}
\end{equation}
\begin{equation}\label{eq:pf4}
mod(f(A)) < \frac{1}{2\pi} \ln{\frac{16}{d}}
\end{equation}
where (\ref{eq:pf2}) follows from the hypothesis $(2)$ above, (\ref{eq:pf3}) follows from the fact that $A$ is a circular annulus and $r\leq \epsilon$,  and (\ref{eq:pf4}) is well-known (see III.A of \cite{Ahl}).\\

\begin{figure}
  \centering
  \includegraphics[scale=0.55]{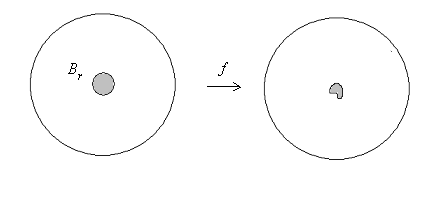}\\
  \caption{The map $f$ in Lemma \ref{lem:qclemnew} is almost-conformal off a small sub-disk.}
  \end{figure}
The rest of the proof is exactly the same as in \cite{Gup1}:\\

Let $\Gamma$ be family of curves between two arcs on the boundary of $\mathbb{D}$, that avoids the set $\mathbb{D}\setminus f(A)$ which by (\ref{eq:pf1}) is of diameter $O(\epsilon)$. Then by a length-area inequality, we have the following estimate on the extremal lengths:
\begin{equation}
1\leq \frac{\lambda_{f(A)}(\Gamma)}{\lambda_\mathbb{D}(\Gamma)} \leq 1 +C_2\epsilon
\end{equation}
where $C_2>0$ depends only on $C_1$.\\
Since this holds for any pair of arcs on the boundary $\partial \mathbb{D}$, it translates to a condition on the cross-ratios of four boundary points, and is enough to prove the extension of $f$ to the boundary is $(1 + C\epsilon)$-quasisymmetric, as claimed (see \cite{AB}, and \cite{Gup1} for details).
 \end{proof}

 \begin{cor}\label{cor:cor1qclem}
 Let $r>0$ be sufficiently small. Suppose $g:\mathbb{D}\setminus B(r) \to \mathbb{D}$ is a conformal embedding that extends to a homeomorphism of $\partial \mathbb{D}$ to $\partial \mathbb{D}$. Then there exists a  $(1+\epsilon)$-quasiconformal map $f:\mathbb{D} \to \mathbb{D}$ such that the extension of $f$ to $\partial \mathbb{D}$ agrees with that of $g$, and 
 \begin{equation}\label{eq:epr}
 \epsilon < 2C^\prime r
 \end{equation}
 for some universal constant $C^\prime>0$.
 \end{cor}
 \begin{proof} Since $g$ is conformal, it is also $(1+ r)$-quasiconformal. By the previous lemma, $g$ extends to a $(1+Cr)$-quasisymmetric map of the boundary, which by the Ahlfors-Beurling extension (see \cite{AB}) extends to an $(1+C^\prime r)$-quasiconformal map of the entire disk, which is our required map $f$. 
 \end{proof}
 
 \begin{cor}\label{cor:corqclem}
Let $\epsilon>0$ be sufficiently small, and $U_0,U$ and $U^\prime$ be conformal disks such that $U_0 \subset U$ and the annulus $A = U\setminus U_0$ has modulus larger than $\frac{1}{2\pi}\ln\frac{1}{\epsilon}$. Then for any conformal embedding $g:A\to U^\prime$ that takes $\partial U$ to $\partial U^\prime$  there is a $(1+C^\prime\epsilon)$-quasiconformal map $f:U\to U^\prime$ such that $f$ and $g$ are identical on $\partial U$. 
\end{cor}
\begin{proof}
By uniformizing, one can assume that $U=U^\prime = \mathbb{D}$ and $U_0 \subset B(r)$ where $r\leq \epsilon$ by the condition on modulus. Hence this reduces to the previous corollary.
\end{proof}

\subsection*{Proof of Lemma \ref{lem:approx1}}
\begin{proof}
Consider the rectangular subsurface $\Sigma_i\subset \Sigma_i^\prime$ and the conformal embedding
\begin{equation*}
g_i:\Sigma_i \to \Sigma\setminus p
\end{equation*}
which exists as the subsurface $\Sigma_i$ is also part of a compact exhaustion of $\Sigma\setminus p$ (see \S6).\\

By construction of $\Sigma_i^\prime$, the complement $\Sigma_i^\prime \setminus \Sigma_i$ is a planar end that is conformally a punctured disk, and by property (3) of the compact exhaustion (see the first section of \S6), so is the complement of $g_i(\Sigma_i)$ in $\Sigma\setminus p$.\\
The conformal embedding  $g_i$ restricts to a conformal map on the  annulus $\Sigma_i \setminus \Sigma_1$. By property (4) of the compact exhaustion (see \S6) this annulus has a modulus $M = A\cdot i$, and hence $M\to \infty$ as $i\to \infty$.\\

By an application of the quasiconformal extension in Corollary \ref{cor:corqclem}, one can get, for sufficiently large $i$, a  $(1 + \epsilon_i)$-quasiconformal map $g_i^\prime$ from the punctured disk $\Sigma_i^\prime \setminus \Sigma_1$ to $U\setminus p$ , that has the same boundary values as $g_i$ on $\partial \Sigma_1$. Here $\epsilon_i\to 0$ as $i\to \infty$. In fact, by (\ref{eq:epr}) and the fact that $M = A\cdot i$, we can derive the better estimate 
\begin{equation}\label{eq:epr2}
\epsilon_i < B 2^{-i}
\end{equation}
where $B>1$ is a universal constant. This will be useful in the next section.\\

Together with the conformal map $g_i$ on $\Sigma_1$, this defines the $(1 + \epsilon_i)$-quasiconformal homeomorphism $f_i:\Sigma_i^\prime \to \Sigma\setminus p $ for all sufficiently large $i$, as required in (\ref{eq:fi}).
\end{proof}

\section{Step 4: A limiting quadratic differential}
In this section and the next we show that after passing to a subsequence the sequence of half-plane surfaces $\Sigma_i^\prime$ converges to a half-plane surface $\Sigma_{n,a}$ -  by \textit{Step 3} this would be conformally equivalent to $\Sigma \setminus p$, as required. In this section we complete a preliminary step, namely we show that after passing to a subsequence the corresponding half-plane differentials converge in  $\widehat{\mathcal{Q}}_m$,  the subset of the bundle of meromorphic quadratic differentials that have a single pole of order exactly $n$ (see Appendix A for a discussion).\\

Recall (from \S6) the half-plane surfaces $\Sigma_i^\prime$ are constructed by excising a disk and gluing in a planar end (a planar domain with a fixed quadratic differential). The idea is that though the glued-in disk gets smaller on the conformal surface, one can still extract some global control on the corresponding half-plane differentials (it is useful to remember that two holomorphic quadratic differentials on a closed surface are identical if they agree on any open set.) In particular, since the planar ends glued in are truncated at heights increasing at a prescribed rate that matches the rate of shrinking of the excised disks  (see the final part of \S7), the restriction of the resulting half-plane differential on a \textit{fixed} conformal disk $U$ converge. We make this precise in the rest of this section.\\

Let $q_i^\prime$ be the half-plane differential corresponding to the half-plane structure on $\Sigma_i^\prime$ and let $U_i = f_i^{-1} (U)$ where $f_i$ is the $(1 + \epsilon_i)$-quasiconformal map in Lemma \ref{lem:approx1}. Let $\phi_i:U_i\to \mathbb{D}$ be the conformal chart mapping $p_i = f_i^{-1}(p)$ to $0$.\\

By the gluing-in construction (see \S6 and the final part of \S7), there is an open set $V_i \subset U_i \subset \Sigma_i^\prime$ containing $p_i$ which, in the metric induced by $q_i^\prime$, is isometric to a  planar end $\mathcal{P}_{H_i}$. Moreover, if $\phi:(U,p)\to (\mathbb{D},0)$ is the conformal coordinate map, then $V_i = f_i^{-1}(V)$ where $V = \phi^{-1}(B(2^{-i}))$. \\

Let us recall that the planar end of $\Sigma_i^\prime$ is isometric (and hence conformally equivalent) to a neighborhood of $0\in \mathbb{C}$ equipped with the following meromorphic quadratic differential (see \S3.3):
\begin{equation}\label{eq:fixqd}
\left(\frac{1}{z^{n+2}} + \frac{ia}{z^{n/2 +2}}\right)dz^2
\end{equation}
Hence the quadratic differential $q_i^\prime$ on $U_i$ (a subset of the planar end) is the pullback by some conformal map, of the above fixed meromorphic quadratic differential on $\mathbb{C}$. The fact that $V_i\subset U_i$ is isometric to the truncation at height $H_i$ of a planar end moreover implies that this conformal map takes $V_i$ to the neighborhood $U_{H_i}$ of $0\in\mathbb{C}$.\\

We shall use the following criterion of convergence of meromorphic quadratic differentials (see the Appendix for a discussion of the proof, and \textit{Criterion $4^\prime$} there):

\begin{lem}\label{lem:convg}
Let $(\Sigma_i, U_i, p_i)$ be a sequence of marked, pointed Riemann surfaces converging to $(\Sigma, U, p)$ in the sense that there exists a $(1+\epsilon_i)$-quasiconformal map $f_i:\Sigma_i\to \Sigma$ that takes $(U_i,p_i)$ to $(U,p)$, such that $\epsilon_i\to 0$ as $i\to \infty$.
Assume that $\Sigma_i$ is equipped with a quadratic differential $q_i$ whose restriction to $U_i$ is the pullback by a univalent conformal map $g_i$ of a fixed meromorphic quadratic differential on $\mathbb{C}$. If $g_i$ form a normal family, then after passing to a subsequence, $q_i\to q\in \widehat{\mathcal{Q}}_m(\Sigma)$.\\
This normality condition is satisfied if for each $i$,  $g_i$ maps the subdomain $V_i\subset U_i$  to $U_{H_i}\subset \mathbb{C}$, where via the conformal identification $\phi_i:(U_i,p_i) \to (\mathbb{D},0)$, we have:
\begin{equation}\label{eq:vib0}
{d}2^{-i}  < dist(0, \partial U_{H_i}) < D2^{-i}
\end{equation}
and
\begin{equation}\label{eq:vib}
{d}2^{-i}  < dist(0, \partial \phi_i(V_i)) < D2^{-i}
\end{equation}
for constants $d,D>0$.
\end{lem}

Note that (\ref{eq:vib0}) holds by Lemma \ref{lem:uhi}. Hence to show that in our case the maps $f_i:\Sigma^\prime_i \to \Sigma\setminus p$ satisfy above the conditions of the above lemma (in the notation already introduced) we only need to prove the bounds (\ref{eq:vib}). Recall here that $V_i$ is the image of a round disk via a quasiconformal map ($f_i^{-1}\circ \phi^{-1}$) of small dilatation. We start with the following more general analytical lemma:

\begin{lem}\label{lem:jyva} Let $\epsilon>0$ be sufficiently small,  and $r$ satisfy
\begin{equation}\label{eq:rbd}
r\geq \frac{\epsilon}{C}
\end{equation}
for some constant $C>1$. Let $f:\mathbb{D}\to\mathbb{D}$ be a $(1+ \epsilon)$-quasiconformal map such that $f(0)=0$.  Let $V = f(B_r)$ be the image of the subdisk of radius $r$ centered at $0$. Then we have
\begin{equation}\label{eq:contain}
\frac{r}{D} \leq dist(0,\partial V) \leq  32r
\end{equation}
for some universal constant $D>0$.
\end{lem}
\begin{proof}
A $K$-quasiconformal self-map of the disk is H\"{o}lder-continuous, with coefficient 16 and exponent $1/K$ (see \cite{Ahl1}). So for any $z\in \partial B_r$ we have:
\begin{equation}\label{eq:hold}
\left| f(z)\right| \leq 16 r^{1/1+\epsilon} <  16 r^{1-\epsilon}  =16r^{-\epsilon} \cdot r \leq 16 (C)^{\epsilon}\epsilon^{-\epsilon} \cdot r < 32r
\end{equation}
for sufficiently small $\epsilon$ (since $C^\epsilon\to 1$ as $\epsilon\to 0$ and $\epsilon^{-\epsilon} < e^{-e} < 1.45$). \\
This gives the inequality on the right, in (\ref{eq:contain}).\\

For the left inequality  of (\ref{eq:contain}) let $z\in \partial B_r$, that is, $\left| z\right| =r$. Then the H\"{o}lder continuity of $f^{-1}:\mathbb{D}\to \mathbb{D}$ yields:
\begin{equation*}
r = \left|z\right| \leq 16\left| f(z) - f(0) \right|^{1/(1+\epsilon)}  = 16\left| f(z)\right|^{1/(1+\epsilon)}
\end{equation*}
and hence we have:
\begin{equation*}
\left| f(z)\right| \geq \left(\frac{r}{16}\right)^{(1+\epsilon)} \geq r \cdot \frac{\epsilon^\epsilon}{C^\epsilon 16^{1+\epsilon}} 
\end{equation*}
where we have used (\ref{eq:rbd}) for the last inequality.\\
It is easy to verify that for sufficiently small $\epsilon$, we have:
\begin{equation*}
\frac{\epsilon^\epsilon}{C^\epsilon16^{1+\epsilon}} \geq \frac{1/2}{2 \cdot 16^{3/2}} 
\end{equation*}
and hence we can take $D = 256$ in (\ref{eq:contain}).
\end{proof}

\textit{Remarks.} 1. The condition (\ref{eq:rbd}) is necessary, as in general quasiconformal maps are H\"{o}lder continuous with exponent less than $1$ and no better, and (\ref{eq:contain}) then fails for small $r$.\\
2. The left inequality can be thought of a quasiconformal version of the Koebe one-quarter theorem, and it would be interesting to give a better estimate of the constant $D$.

\begin{prop}[Step 4]\label{prop:step4}
There is a subsequence of $\{q_i^\prime\}_{i\geq 1}$ that converges to a meromorphic quadratic differential $q$ on $\Sigma$ with a pole of order $n$ and residue $a$ at $p$. 
\end{prop}
\begin{proof} We recapitulate some of the previous discussion in this section:\\
In the construction of the half-plane surface $\Sigma_i^\prime$, one excises a subdisk $V_i\subset U$ and glues it back by a different quasisymmetric map $w$ of a circle (the boundary extension of the conformal map uniformizing $\Sigma\setminus U$ to the rectangular surface $\Sigma_i$)  to form a new conformal surface. The conformal structure on the resulting disk $U_i = (U\setminus V_i) \cup_w V_i$ now admits a uniformizing map $\phi_i:(U_i,p_i)\to (\mathbb{D}, 0)$, and in the quadratic differential metric the disk $\phi_i(U_i)$ is isometric  to a subset of a planar end of residue $a$. It follows that  the quadratic differential  on $\phi_i(U_i)$ is a pullback of the fixed differential (\ref{eq:fixqd}) on $\mathbb{C}$ (with a pole at $0$) via a univalent conformal map. By construction, $U_{H_i}$ corresponds to the subdomain $\phi_i(V_i)$ via this map.\\

We shall verify the convergence criterion of Lemma \ref{lem:convg}. Recall that $\phi: (U,p)\to (\mathbb{D}, 0)$ was the uniformizing map for the (fixed) pointed disk $(U, p)$ on $\Sigma$, and $f_i:(\Sigma_i,U_i)\to (\Sigma,U)$ was a $(1+\epsilon_i)$-quasiconformal map (Lemma \ref{lem:approx1}). Hence the map $f=\phi_i\circ f_i^{-1}\circ \phi^{-1}:\mathbb{D}\to \mathbb{D}$ is a $(1+ \epsilon_i)$-quasiconformal map, and since the disk excised at the beginning of the construction was $\phi^{-1}(B(2^{-i}))$, the image of the disk $B(2^{-i})$ under $f$ is $\phi_i(V_i)$. \\

From  (\ref{eq:epr2}) in the proof of Lemma \ref{lem:approx1}, we have that:
\begin{equation}
\epsilon_i < B 2^{-i}
\end{equation}
for some constant $B>1$, and hence the condition (\ref{eq:rbd}) of Lemma \ref{lem:jyva} is satisfied (here $r=r_i = 2^{-i}$ and $\epsilon=\epsilon_i$). Applying Lemma \ref{lem:jyva} to the map $f$, we then get the  distance bounds (\ref{eq:vib}). The bounds (\ref{eq:vib0}) also hold by Lemma \ref{lem:uhi}, as noted previously. \\

Applying Lemma \ref{lem:convg} to the sequence of half-plane surfaces $\Sigma_i$ with their corresponding half-plane differentials $q_i^\prime$, we have that there is a limiting meromorphic quadratic differential $q\in \widehat{\mathcal{Q}}_m(\Sigma)$. 
\end{proof}

\section{Step 5: A limiting half-plane surface}
 
The set $\mathcal{Q}_{\mathcal{D}^\prime}$ of half-plane differentials associated with the local data 
\begin{center}\label{eq:data}
$\mathcal{D}^\prime  = \{ (n_j, a_j)| n_j\in \mathbb{N}, n_j\geq 4, a_j\in \mathbb{R}_{\geq 0},$  where $a_j=0$ for $n_j$ odd.$\}$
\end{center}
of order of poles and residues at the marked points, is a subset of $\widehat{\mathcal{Q}}_m$. In this section we shall prove that this subset is closed. To simplify the discussion, we shall continue to consider the case of a \textit{single} pole of order $n$ and residue $a$. The proof of the general case follows by an easy extension of the arguments.

\begin{thm}\label{thm:cpct}
Let $\Sigma_i$ be a sequence of half-plane surfaces in $\mathcal{Q}_{\mathcal{D}^\prime}$ such that the corresponding half-plane differentials $q_i \to q \in \widehat{\mathcal{Q}}_m$. Then $q$ is a half-plane differential in $\mathcal{Q}_{\mathcal{D}^\prime}$ .
\end{thm}

This together with Proposition \ref{prop:step4} shall complete the proof of the following:

\begin{prop}\label{prop:step5}
The convergent subsequence of $\{q_i^\prime\}_{i\geq 1}$ converges to a half-plane differential $q$ on $\Sigma$ with a pole of order $n$ and residue $a$ at $p$. 
\end{prop}


\subsection*{Proof of Theorem \ref{thm:cpct}}

Consider a family of half-plane differentials $q_i \in \mathcal{Q}_{\mathcal{D}^\prime}$ that converge in $\widehat{\mathcal{Q}}_m$. The goal is to show that after passing to a subsequence the corresponding sequence of half-plane surfaces converges geometrically to a half-plane surface. This shall be done in this section by considering the \textit{metric spines} (see Definition \ref{defn:mspine}) of the sequence and  constructing the limiting half-plane surface from the limiting metric spine (Definition \ref{defn:limit}).  The geometric convergence is not quite a metric or biLipschitz one, as edges of the spines might collapse. A crucial observation is that by the assumption of convergence in $\widehat{\mathcal{Q}}_m$, in the sequence of metric spines, their embeddings in the surface do not get worse (Lemma \ref{lem:notwist}), cycles do not collapse (Lemma \ref{lem:forest}), and hence the limiting metric spine yields the same marked topological surface. The proof is completed by showing this limiting half-plane surface is indeed a \textit{conformal} limit of the sequence (Lemma \ref{lem:approx2}) by building quasiconformal maps whose dilatation tends to $1$.

\subsubsection*{Spines}

\begin{defn}\label{defn:tspine} A \textit{topological spine} on a surface $S$ with a set of punctures $P$ is an embedded graph that $S\setminus P$ deform-retracts onto. Moreover, we assume each vertex other than the punctures has valence at least three, so there are no unnecessary vertices.
\end{defn}

\begin{defn}\label{defn:mspine} Associated to the half-plane differential $q$ on $\Sigma$ is its \textit{metric spine} $\mathcal{G}(q)$, which is the metric graph on the half-plane surface obtained from the boundaries of the half-planes after identifications. This is a topological spine as in the previous definition - the retraction can be defined by collapsing along the vertical rays on each half-plane.
\end{defn}

The metric spines $\mathcal{G}_i = \mathcal{G}(q_i)$ are topologically equivalent since they are all spines of $\Sigma \setminus P$.  After passing to a subsequence, one can assume that  they are isomorphic. Let $\mathcal{G}_{top}$ denote this fixed finite graph, such that each $\mathcal{G}_i$ is just an assignment of lengths to its set of edges $\mathcal{E}$. By passing to further subsequence we can assume that the edge-lengths $l_i(e)$ for $e\in \mathcal{E}$ converge to a collection of non-negative reals $\{l(e)\}_{e\in \mathcal{E}}$.

\begin{defn}\label{defn:collocus}
The \textit{collapsing locus} $\mathcal{C}$ of the sequence is the set of edges of $\mathcal{G}_{top}$ whose edge-lengths tend to zero, and the \textit{diverging locus} $\mathcal{D}$ of the sequence is the set of edges whose lengths tend to infinity. 
\end{defn}

The goal of the next sections is to show that in fact, after passing to a subsequence, the \textit{embeddings} of these  spines can also be assumed to be the same upto isotopy, that is, the metric spines are identical as \textit{marked} graphs on the surface.  From this it will follow that the collapsing locus  $\mathcal{C}$ has no cycles and $\mathcal{D}$ is empty (Lemma \ref{lem:forest}).

\subsubsection*{Bounded twisting and no collapsing cycles}

\begin{defn}[Twist]  Let $S$ be a surface and  $A\subset S$ an annulus with core a non-trivial simple closed curve $\gamma$. For an embedded arc $\tau$ between the boundary components of $A$, the \textit{twist} of $\tau$ around $\gamma$ relative to $A$ is an integer denoting the number of times $\tau$ goes around $\gamma$ in $A$, upto isotopy fixing the endpoints $\tau \cap \partial A$. (We ignore signs by taking an absolute value.)
\end{defn}

\textit{Remark.} The above twist can be thought of as the distance in the curve complex of the annulus $A$ (see also \S3 in  \cite{Min2}).\\

The following notion is to ensure that \textit{all} non-trivial twists about the core curve are captured in the annulus in the above definition: 

\begin{defn}[Maximal annulus] Given a non-trivial simple closed curve $c$ on a surface, its associated \textit{maximal annulus} $A(c)$ is an embedded open annulus with core curve $c$ such that the  complement of its closure is either empty, or has components which are either disks or once-punctured disks. 
\end{defn}

\textit{Remark.} An example is the maximal embedded annulus realizing the extremal length for $\gamma$ - the complement of its interior is a graph on the surface. In our case, we shall embed the annulus away from disks around the punctures, which gives the once-punctured disks in its complement.

\begin{lem}\label{lem:dtwist} Let $D_c:S\to S$ denote the Dehn twist around a simple closed curve $c$, and $M$ be a positive integer. Then for any maximal annulus $A(c)$ and any simple closed curve $\gamma$ that intersects $c$, the twist of some component arc of $D^n_c(\gamma) \cap A(c)$ about $c$ is greater than $M$ for all $n$ sufficiently large.
\end{lem}
\begin{proof} It suffices to show that the ``twists" of the arcs $D^n_c(\gamma) \cap A(c)$ about $c$ are supported in the interior of $A(c)$, that is, cannot be isotoped away from the annulus. This holds because $A(c)$ is maximal, that is, the complement of its interior comprises  closed disks connected by arcs. The image curve $D^n_c(\gamma)$ being embedded cannot run along these arcs more than once, and any twisting in the interior of the disks can be isotoped to be trivial.\end{proof}

The following finiteness result is well-known. In the statement ``sufficiently large" can be taken to be a finite set of curves consisting of a complete marking (a maximal set of pants curves together with curves intersecting each), or alternatively, the Humphries generators for the mapping class group of $S$.

\begin{lem}\label{lem:toplem1} Let $\mathcal{C}$  be a sufficiently large collection of simple closed curves on a surface $S$. For each $N>0$, the set :
\begin{center}
$\mathcal{S} = \{ \gamma$ is a simple closed curve $|$ Each component of $\gamma \cap A(c)$ has twist less than $N$ around $c$, for each $c\in \mathcal{C}\}$ 
\end{center}
is a finite set.
\end{lem}
\begin{proof}
There are finitely many curves $\gamma_1, \gamma_2,\ldots \gamma_k$ on $S$ upto the action of the mapping class group $MCG(S)$, which is virtually generated by Dehn twists around $\mathcal{C}$.  By Lemma \ref{lem:dtwist} powers of a Dehn twist around $c \in \mathcal{C}$ increases the twist of some component of $\gamma_j \cap A(c)$ around $c$, and hence by the condition that twists are bounded there are only finitely many mapping classes $g_1,\ldots g_N$  such that $g_i\cdot \gamma_j \in \mathcal{S}$ (where $1\leq i\leq N$, and $1\leq j\leq k$). Hence $\mathcal{S}$ is finite.
\end{proof}

As a consequence we have the finiteness of spines with a similar ``bounded twisting" condition:

\begin{lem}\label{lem:toplem2} Let $\mathcal{C}$  be a sufficiently large collection of simple closed curves on a surface $S$. For each $N>0$, the set :
\begin{center}
$\mathcal{M} = \{ m$ is a topological spine for $S\setminus P$ $|$ Each component of $e\cap A(c)$ has twist less than $N$ around $c$, for each edge $e\in m$ and $c\in \mathcal{C}\}$ 
\end{center}
is a finite set.
\end{lem}
\begin{proof}
It suffices to show that each cycle in a spine in $\mathcal{M}$ corresponds to finitely many possible homotopy classes of curves on the surface. Since there are uniformly bounded number of edges in the spine (depending only on the topology of the surface), and each edge of the cycle has bounded twisting around each $c\in\mathcal{C}$, so does the embedded cycle on the surface, and the finiteness follows from the previous lemma. \end{proof}

Consider now the sequence of half-plane surfaces $\Sigma_i$ and the metric spines $\mathcal{G}_i$. 

\begin{lem}\label{lem:apriorib} There exists a choice of  disk neighborhoods $U_i\subset \Sigma_i$ around the pole of $q_i$,  for each $i\geq 1$, such that for the sequence of singular flat surfaces $S_i = \Sigma_i \setminus U_i$  we have:\\
(1) The sequence of areas $Area(S_i)  = \int\limits_{S_i} \lvert q_i\rvert$ is uniformly bounded from above.\\
(2) For any simple closed curve $\gamma$ on $S$, the length in the $q_i$-metric of any curve in $S_i$ homotopic to $\gamma$ is uniformly bounded from below.
\end{lem}

\begin{proof}
Since the meromorphic quadratic differentials $q_i$ are converging in $\widehat{\mathcal{Q}}_m$, the underlying conformal structures converge. Choose a disk $U$ around the pole in $\Sigma$, and for a choice of $\epsilon>0$, fix a sequence of $(1+ \epsilon)$-quasiconformal maps $f_i:\Sigma_i\to \Sigma$ preserving the puncture. Set $U_i = f_i^{-1}(U)$. By the convergence of $q_i$, the area or $L^1$-norm of $q_i$ on $S_i$ converges to $Area(\Sigma\setminus U)$, and hence we have statement (1) above.\\
Also, by the convergence, the singular flat $q_i$-metrics lie in a compact set, and hence so do the lengths of the  geodesic representatives of a fixed simple closed curve $\gamma$ on $\Sigma_i$. (By properties of non-negative curvature, such a geodesic representative is unique except the case when they sweep out a flat annulus,  in which case the lengths are all the same.) This implies that the length of any curve in $S_i$ homotopic to $\gamma$ is uniformly bounded from below, which is statement (2). 
\end{proof}

\begin{lem}\label{lem:notwist}
For any simple closed curve $\gamma$, consider a maximal embedded annulus $A_i(\gamma)$ on $S_i$. Then for any edge $e$ of the spine $\mathcal{G}_i$, each component of $e\cap A_i(\gamma)$ has uniformly bounded length as $i\to \infty$.  Moroever, $A_i(\gamma)$ is a maximal annulus on $\Sigma_i$ and $e\cap A_i(\gamma)$ has uniformly bounded twist about $\gamma$.
\end{lem}
\begin{proof}
The horizontal edge $e$ from the spine twists across the annular region $A_i(\gamma)$. Assume a large number of twists. For each point in $e\cap A_i(\gamma)$ sufficiently in the middle, there is a vertical segment into an adjacent half-plane, which because of the twisting, has length at least the circumference of $A_i(\gamma)$ before it can escape the annulus (see Figure 20).\\

\begin{figure}[h]
  \centering
  \includegraphics[scale=0.8]{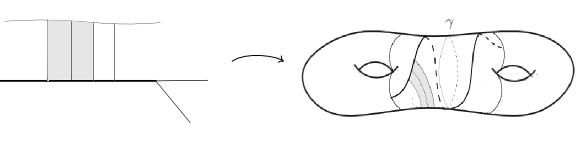}\\
  \caption{A collar about the metric spine eating into an adjacent half-plane (shown on the left) embeds in $A_i(\gamma)$ contributing to area.}
  \end{figure}

This circumference is bounded below by (2) of Lemma \ref{lem:apriorib}, and hence this sweeps out a definite metric collar in $A_i(\gamma)$ around the spine, which contributes area proportional to the length of $e\cap A_i(\gamma)$. On the other hand, by (1) of Lemma \ref{lem:apriorib} the areas of $A_i(\gamma)$ (which are less than $Area(S_i)$) remain remain uniformly bounded, and hence so do the lengths of $e\cap A_i(\gamma)$. Again by the uniform lower bound on the circumferences of $A_i(\gamma)$, this implies that the twisting of $e$ around $\gamma$ cannot tend to infinity (each twist will add to a length of at least half the circumference). Since the complement of a $S_i$ is a punctured disk $U_i$, the annulus $A_i(\gamma)$ that is maximal on $S_i$, is also maximal on $\Sigma_i$.
\end{proof}

The following is now immediate from Lemma \ref{lem:toplem2}:

\begin{cor}\label{cor:marking}
After passing to a further subsequence, we can assume that the metric graphs $\mathcal{G}_i$  are isomorphic as marked spines on the surface. In particular, a cycle in the graph corresponds to the same homotopy class of a closed curve throughout the sequence.
\end{cor}

\begin{lem}[No collapsing cycles]\label{lem:forest}
$\mathcal{C}$ is a forest, that is, it contains no cycle, and $\mathcal{D}$ is empty.
\end{lem}
\begin{proof} The meromorphic quadratic differentials $q_i$ lie in a compact set  $K$ of $\widehat{\mathcal{Q}}_m$ since they form a convergent sequence. By Corollary \ref{cor:marking}, after passing to a subsequence a cycle in the metric graph $\mathcal{G}_i$ corresponds to a (fixed) non-trivial curve in $\Sigma$ whose lengths in the singular flat $q_i$-metric  must have a uniform lower bound by the compactness of $K$. By the uniform (upper) length bound of Lemma \ref{lem:notwist} there cannot be an edge whose lengths tend to infinity.\end{proof}

Recall from the discussion preceding Definition \ref{defn:collocus} that the finite edge-lengths of the metric spines converge after passing to a subsequence. Let  $\mathcal{G}$ be metric graph obtained by assigning this length $l(e)$ to every edge $e\in \mathcal{G}_{top}$, where it is understood that any component tree of the collapsing locus $\mathcal{C}$ is identified with a single vertex. The previous  lemma ensures that  $\mathcal{G}$  has the same topological type, that is, remains homotopy equivalent to $\mathcal{G}_{top}$.

 \begin{figure}[h]
  \centering
  \includegraphics[scale=0.5]{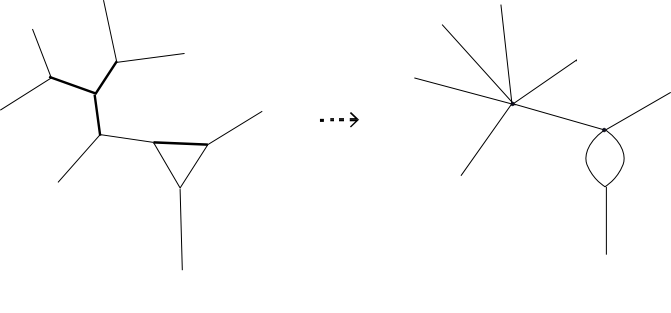}\\
  \caption{The edges in $\mathcal{C}$ (shown in bold on the left) collapse along the sequence of graphs $\mathcal{G}_i$. }
  \end{figure}

\begin{defn}[Limiting half-plane surface]\label{defn:limit}
The half-plane surface $\Sigma_{\mathcal{G}}$ is defined to be the one obtained by gluing in $n$ half-planes along the metric graph  $\mathcal{G}$  in the combinatorial order identical to the gluing of the half-planes along the metric spine of each $(\Sigma, q_i)$.  (Here we assume we have passed to a subsequence where the metric spines are identical as marked graphs.)
\end{defn}

 \subsubsection*{Proving $\Sigma_\mathcal{G}$ is the conformal limit}

\begin{lem}\label{lem:approx2}
For all sufficiently large $i$ there exist $(1+\epsilon_i)$-quasiconformal maps
\begin{equation*}
\overline{h_i}:\Sigma_i^\prime \to \Sigma_\mathcal{G}
\end{equation*}
where $\epsilon_i\to 0$ as $i\to \infty$.
\end{lem}

We start by observing that a map that collapses a short segment of the boundary of a half-plane $\mathbb{H}$ can be extended to a map of the half-plane that is almost-isometric away from a suitable neighborhood:

\begin{lem}\label{lem:extend}
For any $\epsilon>0$ there exists a $\delta >0$ such that a map $h:\partial \mathbb{H}\to \partial \mathbb{H}$ that collapses an interval $I$ of length $\epsilon$ and isometric on its complement, extends to a  homeomorphism $\overline{h}:\mathbb{H}\to  \mathbb{H}$ that is an isometry away from an $\delta$-neighborhood of $I$. Moroever, one can choose $\delta$ such that $\delta\to 0$ as $\epsilon\to 0$.
\end{lem}
\begin{proof} One can in fact choose $\delta = \epsilon$, and the extension to be height-preserving, as follows.  Let $\mathbb{H} = \{(x,y)|$ $y>0\}$ and $I = [-\epsilon/2,\epsilon/2]\times\{0\}$. Choose a ``bump" function $\phi(x,y)$ supported on the rectangle $R= [-\epsilon,\epsilon]\times[0,\epsilon]$ that is positive in its interior and interpolates between $0$ on $I$ and $1$ on the complement of $R$. This function $\phi(x,y)$ is the ``dilatation factor" of a horizontal stretch map
\begin{equation*}
h(x,y) = \left(\phi(x,y)x, y\right)
\end{equation*}
that is the required extension.
\end{proof}

\begin{proof}[Proof of Lemma \ref{lem:approx2}]

For any $\epsilon>0$, we shall construct a $(1+\epsilon)$-quasiconformal map from $\Sigma_i^\prime$ to $ \Sigma_\mathcal{G}$ for all sufficiently large $i$. The construction is in two steps: in the first step we map to an intermediate half-plane surface $\Sigma_{\mathcal{G}^{\prime}_i}$.\\

\textit{Step I.}
Let $\mathcal{E}$ be the set of edges in $\mathcal{G}_{top}$. The cardinality $\vert \mathcal{E}\vert$ is finite, a number depending only on the genus of $\Sigma$. Recall also (see Definition \ref{defn:collocus}) that  $\mathcal{C} \subset \mathcal{E}$  is the sub-graph consisting of edges whose lengths along the sequence $\mathcal{G}_i$ tend to zero.\\

The lengths of edges of $\mathcal{E}\setminus \mathcal{C}$ in  $\mathcal{G}_i$ however converge to \textit{positive} lengths of the corresponding edges in $\mathcal{G}$. Hence for all $e\in \mathcal{E}\setminus \mathcal{C}$ we have:
\begin{equation}\label{eqn:ratio}
r_i=  \frac{l_i(e)}{l(e)} \to 1
\end{equation}
as $i\to \infty$.\\

Consider the metric graph $\mathcal{G}^{\prime}_i$ obtained by assigning the following lengths to the edges of $\mathcal{G}_{top}$: $l(e)$ to all edges in $\mathcal{E}\setminus \mathcal{C}$ and $l_i(e)$ to all edges in  $\mathcal{C}$. We can construct a $K^1_i$-biLipschitz map 
\begin{equation}\label{eq:h1i}
h^1_i:\mathcal{G}_i\to \mathcal{G}^{\prime}_i
\end{equation}
that preserves vertices, is a linear stretch map on all the finite-length edges in $\mathcal{E}\setminus \mathcal{C}$, and is an isometry on every other edge. By (\ref{eqn:ratio}), the stretch-factors, and therefore the biLipschitz constants  $K^1_i\to 1$ as $i\to \infty$.\\

Note that any $K$-biLipschitz map
\begin{equation*}
b:\mathbb{R}\to \mathbb{R}
\end{equation*}
can be extended to a $K$-biLipschitz map $\bar{b}:\mathbb{H}\to\mathbb{H}$ of the upper half-plane (here $\mathbb{R} = \partial \mathbb{H}$ by mapping
\begin{equation*}
(x,y) \mapsto (b(x), y).
\end{equation*}

Applying this to the map (\ref{eq:h1i}) above, we get for all sufficiently large $i$, a $(1+\epsilon)$-biLipschitz extension 
\begin{equation}\label{eq:ext1}
\overline{h^1_i}:\Sigma_i\to \Sigma_{\mathcal{G}^{\prime}_i}
\end{equation}
where $ \Sigma_{\mathcal{G}^{\prime}_i}$ is the half-plane surface obtained by gluing half-planes along $\mathcal{G}^{\prime}_i$.\\

\textit{Step II.}
Note that the graph $\mathcal{G}$ is obtained by collapsing all the edges of $\mathcal{C}$ in $\mathcal{G}^{\prime}_i$. 

Let $E$ be the set of finite-length edges of $\mathcal{E}\setminus \mathcal{C} $ and consider the minimum length
\begin{equation*}
c = \min\limits_{e\in E} l(e)
\end{equation*}
if $E\neq \emptyset$ is non-empty, and set $c=2$ if $E = \emptyset$.\\

\begin{figure}[h]
  \centering
  \includegraphics[scale=0.75]{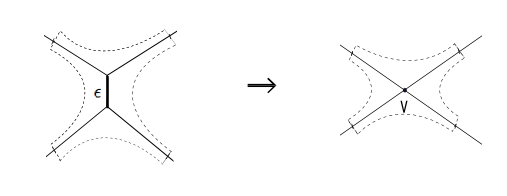}\\
  \caption{In \textit{Step II} the map between half-plane surfaces that is an isometry away from a neighborhood of $\mathcal{C}$ is finally adjusted to an almost-conformal map. }
  \end{figure}

If $\mathcal{C}_i$ denotes the metric subgraph corresponding to $\mathcal{C}$ in $\mathcal{G}_i$, recall we have
\begin{equation}\label{eqn:seq}
diam(\mathcal{C}_i) \to 0
\end{equation}
as $i\to\infty$.\\

Since $\mathcal{C}_i$ is a forest by Lemma \ref{lem:forest}, for sufficiently large $i$ we have:\\
(1) For each half plane $\mathbb{H}$ of the half-plane surface  $ \Sigma_{\mathcal{G}^{\prime}_i}$ the intersection $\partial \mathbb{H} \cap \mathcal{C}_i$ is a union of segments each of length less than $\epsilon$ and separated by a distance at least $c$.\\
(2) The $c/2$-neighborhood $N_{c/2}$ of $\mathcal{C}_i$ is topologically a union of disks, one for each component tree.\\

Moreover, for sufficiently large $i$, $\epsilon$ is small enough such that the corresponding $ \delta <c/2$ where $\delta$ is as in Lemma \ref{lem:extend}. By an application of that Lemma on each half-plane, one can build a homeomorphism
\begin{equation}\label{eq:h1i0}
\overline{h^\prime_i}:\Sigma_{\mathcal{G}^{\prime}_i}\to \Sigma_\mathcal{G}
\end{equation}
that is an isometry (and hence conformal) away from a $\delta$-neighborhood $N_\delta$ of $\mathcal{C}_i$.\\

By observation (2) above each component of $N_{c/2}\setminus N_\delta$ is topologically an annulus, and has modulus that tends to infinity as $\delta \to 0$. Hence for sufficiently large $i$, one can apply the quasiconformal extension of Corollary \ref{cor:corqclem} to adjust the map $\overline{h^\prime_i}$ in the component disks of $N_\delta$ to obtain a $(1 + C^\prime\epsilon)$-quasiconformal map
homeomorphism
\begin{equation}\label{eq:h1i1}
\overline{h^2_i}:\Sigma_{\mathcal{G}^{\prime}_i}\to \Sigma_\mathcal{G}
\end{equation}

Recall the map (\ref{eq:ext1}) from $Step$ I. The composition 
\begin{equation}\label{eq:h1i2}
\overline{h_i} =\overline{h^2_i}\circ \overline{h^1_i}:\Sigma_i \to \Sigma_\mathcal{G}
\end{equation}
is then $(1+ C^{\prime\prime}\epsilon)$-quasiconformal, for some universal constant $C^{\prime\prime}>0$, as required.
\end{proof}

\section{Step 6: Prescribing the leading order term }

 From Lemmas \ref{lem:approx1} and \ref{lem:approx2} we now have:
 
 \begin{prop}\label{prop:qcseq} The quasiconformal maps
$g_i\circ f_i^{-1}: \Sigma \setminus p \to \Sigma_\mathcal{G}$
have quasiconformal dilatation that tends to $1$ as $i\to \infty$, and hence limits to a conformal homeomorphism 
\begin{center}
$g:\Sigma \setminus p \to \Sigma_\mathcal{G}$
\end{center}
in the sense of uniform convergence on compact sets, after passing to a subsequence. Here, $\Sigma_\mathcal{G} = \Sigma_{n,a}$, that is, it is a half-plane surface with a pole of order $n$ and residue $a$. Moroever, $g$ is homotopic to the identity map.
\end{prop}
\begin{proof}
The quasiconformal maps extend to a map between the closed surfaces, mapping $p$ to $\infty$. The uniform convergence is a standard application of the compactness of a family of quasiconformal maps with fixed domain and target, and bounded dilatation. Since the quasiconformal dilatations of $f_i$ and $g_i$ tend to $1$ as $i\to \infty$,  that the limiting homeomorphism is $1$-quasiconformal, and hence conformal.\\

By construction (see Proposition \ref{prop:step4}), the limiting half-plane differential has a pole of order $n$ and residue $a$.
Inspecting the construction of the quasiconformal homeomorphisms $f_i$ and $g_i$, we observe that both are homotopic to the identity ($f_i$ is a quasiconformal map on a disk together with the identity map on its complement, and $g_i$ restricts to  a homotopy equivalence of the metric spines). Hence so is each homeomorphism $g_i\circ f_i^{-1}$, and the limit $g$.
\end{proof}

To complete the proof of Theorem \ref{thm:main} (in the case of a single marked point $p$) it only remains to show that the half-plane differential has a leading order term $c$ with respect to the fixed choice of coordinate neighborhood $U$ around $p$. By Lemma \ref{lem:pullb} it will be enough to show that the above conformal map $g$ has derivative of suitable magnitude with respect to this conformal neighborhood. This is where a suitable choice of the constant $H_0$ in (\ref{eq:hi}) will be made.\\

As before let $\phi:U\to \mathbb{D}$ be a conformal homeomorphism such that $\phi(p)=0$. Let $g(U)$ be a subset of a planar end is identified with a complement of a compact set in $\mathbb{C}$. As in \S3.3, by an inversion map, one has a conformal homeomorphism $\psi:g(U) \to V$  that takes $\infty$ to $0$, where $V$ is a simply connected domain in $\mathbb{C}$ containing $0$. In the rest of this section we shall show:

\begin{prop}\label{prop:deriv} There is a choice of $H_0$ in (\ref{eq:hi}) for which the conformal map $G = \psi\circ g\circ\phi^{-1}:\mathbb{D}\to \mathbb{C}$ that takes $0$ to $0$, has derivative $\left| G^\prime(0)\right| = c^{-\frac{1}{n-2}}$. For this $H_0$, we have that the leading term of the half-plane differential for $\Sigma_\mathcal{G}$ is $c$ with respect to the coordinate chart $U$.
\end{prop}

The proof of this needs the following analytical lemma:

\begin{lem}\label{lem:qclem2}
Let $f_i:\mathbb{D}\to \mathbb{C}$ be a sequence of quasiconformal embeddings such that:\\
(1) $f_i(0)=0$ for all $i$,\\
(2) $f_i$ is $(1+\epsilon_i)$-quasiconformal, where $\epsilon_i\to 0$ as $i\to \infty$, and\\
(3) for some sequence $r_i\to 0$ we have that $f(B_{r_i}) = V_i$, where $V_i$ is an open simply connected domain containing $0\in \mathbb{C}$ having a uniformizing map $\phi_i:(V_i,0) \to (\mathbb{D},0)$, such  that:
\begin{equation}\label{eq:dl}
\left|r_i \phi_i^\prime(0)\right| \to \alpha
\end{equation}
as $i\to \infty$.\\
Then after passing to a subsequence, $f_i$ converges uniformly to $f$, a univalent conformal map such that $\left|f^\prime(0)\right| = 1/\alpha$.
\end{lem}
\begin{proof}
It is a standard fact that a sequence of $K$-quasiconformal self maps of $\hat{\mathbb{C}}$ normalized by the additional requirement that it fixes two points $0$ and $\infty$ forms a sequentially compact family with respect to uniform convergence. This is satisfied by the family $\{f_i\}$  for each $K>1$ and hence there is a limiting \textit{conformal} map $f$ as required, which is either univalent or constant. It only remains to show that $\left|f^\prime(0)\right| = 1/\alpha$ - this will also rule out the latter possibility.\\

Consider the conformal dilatation $\psi_i:\mathbb{D} \to B_{r_i}$ where $\psi_i(z) = r_iz$. Note that $\psi_i^\prime(0) = r_i$. Then the composition $F_i=\phi_i\circ f_i\circ \psi_i:\mathbb{D}\to \mathbb{D}$ is  $(1 + \epsilon_i)$-quasiconformal, where $\epsilon_i\to 0$ as $i\to \infty$ by property (2) above. The compactness result  mentioned above (see also Theorem 1 of \cite{Ahl}) implies that there is a subsequence that converges uniformly to a conformal map $F$ . Moreover, since each map along the sequence preserves the point $0\in \mathbb{D}$, so does the limit and by the Schwarz lemma, $F$ is a rotation and in particular
\begin{equation}\label{eq:psip}
\left| F^\prime(0) \right|=1.
\end{equation}
If each $f_i$ were differentiable at $0$, and the sequence of derivatives converged to the derivative of the limit, we would have by using the chain rule:
\begin{equation*}
 \left| F^\prime(0) \right| = \lim\limits_{i\to \infty} \left| (\phi_i\circ f_i\circ \psi_i)^\prime(0) \right| =  \lim\limits_{i\to \infty} \left| \phi_i^\prime(0) \right| \cdot  \left| {f_i}^\prime(0) \right| \cdot r_i= \alpha  \left|  f^\prime(0) \right| 
\end{equation*}
where the last equality is from (\ref{eq:dl}). By (\ref{eq:psip}) the argument would be complete, namely $\left|f^\prime(0)\right| = 1/\alpha$ as desired. However, $f_i$ are merely quasiconformal, and may not be differentiable at $0$. However they have derivatives which are defined almost-everywhere and are locally integrable, and that converge in norm to the derivative of the limit. So we run the above argument with the \textit{averages} ($L^1$-norms) of the total derivatives in a sequence of shrinking disks around $0$:\\

Let $B_\delta$ be a disk around $0$ of radius $\delta>0$ (sufficiently small).  For a conformal or quasiconformal map $F$ defined on $\mathbb{D}$ we have a nonnegative real number
\begin{equation*}
D_{avg}^\delta F = \frac{1}{Area(B_\delta)}\int\limits_{B_\delta} \lvert DF_i(z)\rvert dzd\bar{z}.
\end{equation*}

In what follows we shall pass to a converging subsequence wherever needed.\\
The sequence of quasiconformal maps  $F_i:\mathbb{D}\to \mathbb{D}$ converges uniformly to a rotation, and hence 
\begin{equation} \label{eq:davg1}
D_{avg}^\delta F_i\to 1
\end{equation}
as $i\to \infty$.\\
Also, the sequence of quasiconformal maps $f_i:\mathbb{D}\to\mathbb{C}$ converges uniformly to a conformal map $f$, and this implies that 
\begin{equation}\label{eq:davg2}
D_{avg}^\delta f_i\ \to \left|f^\prime(0)\right|
\end{equation}
as $i\to \infty$ and $\delta\to 0$.\\
Further, for the univalent conformal map  $\phi_i:V_i\to \mathbb{D}$, one has that
\begin{equation}\label{eq:phiex}
\phi_i^\prime (z) = \phi_i^\prime (0) + Az + O(z^2)
\end{equation}
where $\left|A\right|$ is bounded above by a universal constant (a standard fact in univalent mappings).\\

Since quasiconformal maps are differentiable almost everywhere,  we have by the chain rule that
\begin{equation}\label{eq:break}
\left| DF_i(z) \right| = \left| D\phi_i(z)Df_i(z) D\psi_i(z) \right| =  r_i \left| Df_i(z) \right| \left|\phi_i^\prime(z)\right|
\end{equation}
for $z\in B_\delta^\prime$, a full-measure subset of $B_\delta$.\\

By (\ref{eq:dl}) for any $\epsilon>0$ there is $i$ sufficiently large such that  $\alpha -\epsilon < r_i\left|\phi_i(0)\right| < \alpha +\epsilon$ and  $Ar_i <\epsilon$ where $A$ is the constant in (\ref{eq:phiex}). Using (\ref{eq:break}) we then have:
\begin{equation*}
\int\limits_{B_\delta^\prime} \lvert DF_i(z)\rvert dzd\bar{z} = \int\limits_{B_\delta^\prime}   \left| Df_i(z) \right| r_i\left|\phi_i^\prime(z)\right| dzd\bar{z} \leq \int\limits_{B_\delta^\prime}  \left| Df_i(z) \right| \left(\alpha + 2\epsilon\right)dzd\bar{z} 
\end{equation*}
where we have used (\ref{eq:phiex}) and the above observations for the last inequality.\\

We have a similar bound from below.  By taking $i\to \infty$ we get that 
\begin{equation*}
\left|\alpha D_{avg}^\delta f_i -   D_{avg}^\delta F_i \right| \to 0
\end{equation*}
and hence  by (\ref{eq:davg1})  we have:
\begin{equation*}
D_{avg}^\delta f_i \to 1/\alpha 
\end{equation*}
which by (\ref{eq:davg2}) is equal to $\left|f^\prime(0)\right|$ as $\delta \to 0$. This completes the proof. \end{proof}

In our case, we shall apply Lemma \ref{lem:qclem2} to the sequence of quasiconformal maps $g_i\circ f_i^{-1}:\Sigma \to \Sigma_{\mathcal{G}}$ (see Propn. \ref{prop:qcseq}) after restricting to the open set $U$ that we conformally identify with $\mathbb{D}$ (see the discussion preceding Propn. \ref{prop:deriv}).  Recall that by construction the above quasiconformal map takes the open set $V_i\subset U$ to a planar end $\mathcal{P}_{H_i}$ that can be identified with an open neighborhood $U_{H_i}$ of $0\in \mathbb{C}$ (See also Lemma \ref{lem:uh}).\\

To use  Lemma \ref{lem:qclem2}, we need to verify the condition (3).  Recall from \S3.3 that $U_{H_i}$ is the image of the planar end truncated at height $H_i$. Consider a sequence of uniformizing conformal maps $\phi_i:  U_{H_i}\to \mathbb{D}$ for $i\geq 1$, and recall that $r_i= 2^{-i}$ by construction (\S6). In this setup, we have:

\begin{cor}\label{cor:dlim}  After passing to a  subsequence, $r_i\left| \phi_i^\prime(0)\right|$ has a limit $L$ as $i\to \infty$. Moreover, $L\to \infty$ as $H_0\to \infty$, and $L\to 0$ as $H_0\to 0$, where $H_0>0$ was the constant chosen arbitrarily in (\ref{eq:hi}), and $L$ varies continuously with $H_0$.
\end{cor}
\begin{proof}
By construction (see the first part of \S6) we have $r_i= 2^{-i}$ and $H_i^{2/n} = H_0\cdot 2^i$ (see also (\ref{eq:hi})). Corollary \ref{cor:uhmap} now shows that the sequence  $\{r_i\left| \phi_i^\prime(0)\right|\}_{i\geq1}$ lies in the interval $[H_0/4D_2, H_0/D_1]$, and hence after passing to a subsequence converges to $L$ in that interval.  As one varies $H_0$, all domains (and hence the conformal maps $\phi_i$ and its derivatives) vary continuously (see also Lemma \ref{lem:mon}) and the same subsequence converges to a value that varies continuously  (and lies in an appropriately shifted interval).
\end{proof}

\begin{proof}[Proof of Proposition \ref{prop:deriv}]
By the previous corollary, there exists a choice of $H_0$ in (\ref{eq:hi}) for which 
\begin{equation*}
r_i\left| \phi_i^\prime(0)\right| \to c^\frac{1}{n-2}
\end{equation*}
Applying Lemma \ref{lem:qclem2} to the sequence of quasiconformal maps
\begin{equation*}
\psi\circ g_i\circ f_i^{-1} \circ\phi^{-1}:\mathbb{D} \to \mathbb{C}
\end{equation*}
it follows that $G= \psi\circ g\circ \phi^{-1}$ has derivative $c^{-\frac{1}{n-2}}$ at $0$.\\

Now, the half-plane differential on the disk $\mathbb{D}$ uniformizing $U$ is the pullback of the standard meromorphic differential on $\mathbb{C}$ (see (\ref{eq:stan})) by this map $G$. Since the leading order term at the pole for this standard differential is $1$, by Lemma \ref{lem:pullb}, the leading order term is of the half-plane differential is $\left|G^\prime(0)\right|^{2-n}=c$ , as required.
\end{proof}

\section{Summary of the proof}

Collecting the results of the previous sections, we have:
\begin{proof}[Proof of Theorem \ref{thm:main}]

Propositions \ref{prop:qcseq} and \ref{prop:deriv} complete the proof of Theorem \ref{thm:main} in the case of a single marked point. This easily generalizes to the case of multiple poles, as we briefly summarize:\\

For a set $P = \{p_1,\ldots p_n\}$ on $\Sigma$ consider fixed coordinate charts $U_1,\ldots U_n$ around each. As in \S6 , we consider a compact exhaustion of the surface by excising  subdisks of radii $r_i$ tending to zero, from each (more specifically, we choose $r_i = 2^{-i}$). For each compact subsurface $\Sigma_i$ we choose a number of arcs on each boundary component depending on the desired orders of poles and by the quadrupling construction we construct a compact Riemann surface with a collection of distinct homotopy classes of curves (the assumption that $\Sigma \neq \widehat{\mathbb{C}}$ if $n=1$ ensures that the homotopy classes are distinct).  We prescribe a Jenkins-Strebel differential with closed trajectories in these homotopy classes and given cylinder circumferences on this surface. Moreover, by Lemma \ref{lem:arcs} one can choose the arcs so that the extremal lengths of these curves are precisely what is needed for the cylinder lengths to be 
\begin{equation*}
H_i^j=\left( H_0^j\cdot 2^i\right)^{n_j/2}
\end{equation*}
for the $j$-th marked point ($1\leq j\leq n$), where $n_j$ is the desired order of the pole at that marked point. (See \ref{eq:hi}) in \S6.) By quotienting back, one gets a ``rectangular"  metric on $\Sigma_i$ with polygonal boundaries of prescribed dimensions,  that can be completed to a half-plane surface $\Sigma_i^\prime$ by gluing in $n$ planar ends.  As in Lemma \ref{lem:uhi}, the above prescribed dimensions of the polygonal boundary ensures that the planar end glued in is conformally a disk of radius $O(r_i)$ around $0\in \mathbb{C}$.\\

The fact that the conformal structures on $\Sigma \setminus P$ and $\Sigma_i^\prime$ differ only a union of disks of small radii  ($r_i$) allows us to build an almost conformal homeorphism between them (\S8) that tends to a conformal homeomorphism as $i\to \infty$. The geometric control on the dimensions of polygonal boundaries of the rectangular surfaces $\Sigma_i$ also shows that the corresponding half-plane differentials converge to a meromorphic quadratic differential on $\Sigma$ with poles at $P$ (as in \S 9), and as in \S10 this limiting differential is in fact half-plane. Finally, as in \S11 one can show that an appropriate choice of  ``scaling" factors (the constants $H_0^j$ above) while constructing the sequence of rectangular surfaces, ensures that the leading order terms at each pole of the resulting half-plane differential are the desired real numbers.
\end{proof}

\section{Applications and questions}

\subsection{An asymptoticity result}
We had previously shown (see \cite{Gup1} for  a precise statement):

\begin{jthm}[\cite{Gup1}] A grafting ray in a generic direction in Teichm\"{u}ller space is strongly asymptotic to some Teichm\"{u}ller ray.
\end{jthm}

The main result of \cite{Gup25} is a generalization of the above asymptoticity result to \textit{all} directions.
The idea of the proof is to consider the conformal limit of the grafting ray, and find a conformally equivalent singular flat surface that shall be the conformal limit of the corresponding Teichm\"{u}ller ray. The strong asymptoticity is shown by adjusting this conformal map to almost-conformal maps between surfaces along the rays. \\

The result of this paper is used to find the conformally equivalent singular flat surface mentioned above: namely, Theorem \ref{thm:main} can be generalized easily to include poles of order $2$, which correspond to half-infinite cylinders in the quadratic differential metric. This is used to obtain a  (generalized) half-plane surface $Y_\infty$ and a conformal map
\begin{equation*}
g:X_\infty \to Y_\infty
\end{equation*}
for each component $X_\infty$ of the conformal limit of the grafting ray.\\
 Prescribing the leading order term, or equivalently the derivative of the above conformal map $g$, helps to construct the controlled quasiconformal gluings of truncations of these infinite-area surfaces. 

\subsection{The question of uniqueness}
The construction of single-poled half-plane differentials on surfaces (see \S4.4) proceeds by introducing a slit in the metric spine of a single-poled hpd  $\hat{\mathbb{C}}$ as above, and gluing by an interval-exchange. (This does not affect the residue and leading order coefficient at the pole.)\\

Since there are non-unique choices of single-poled hpds (of order greater than $4$) on $\hat{\mathbb{C}}$ which have the same residue and leading order coefficient (see \S4.2), one can see that uniqueness does not hold in general, in Theorem \ref{thm:main}.\\

However, we conjecture:

\begin{conj} When all the orders of the poles are $4$, the half-plane differential with prescribed residues and leading order terms that exists by Theorem \ref{thm:main}, is unique.
\end{conj}

and more generally one can ask:

\begin{ques} Does uniqueness of the half-plane differential hold if one prescribes further local data, in addition to the order of pole, residue and leading order term?
\end{ques}

\subsection{Limits of Teichm\"{u}ller rays}
One obtains half-plane surfaces as geometric limits of Teichm\"{u}ller rays (details in the forthcoming paper \cite{Gup25}). Roughly speaking, along a Teichm\"{u}ller geodesic  there is a stretching of the quadratic-differential metric in the ``horizontal" direction (after rescaling we can assume that distances in the vertical direction remains unchanged). This increases the area of the surface monotonically, and stretches a neighborhood of the critical graph of vertical saddle-connections, to a half-plane surface.\\

Generically (for directions determined by \textit{arational} laminations) one obtains a collection of $2g-2$ half-plane differentials $(\mathbb{C}, zdz^2)$, and more interesting limits are obtained for directions determined by \textit{non-filling} laminations. \\

One of the questions to be addressed in forthcoming work is:

\begin{ques} Given a collection of half-plane surfaces with pairings of poles with matching order and residues, is it possible to construct a Teichm\"{u}ller ray with that (disconnected) half-plane surface as a limit?
\end{ques}

\appendix
\section{Criteria for convergence}
Let $\Sigma$ be a Riemann surface with marked points $p_1,p_2,\ldots p_n$, and $k_1,k_2,\ldots k_n$ a tuple of integers such that $k_j\geq 4$. Let $\widehat{\mathcal{Q}}_m(\Sigma)$ be the  meromorphic quadratic differentials on $\Sigma$ with a pole of order $k_j$ at $p_j$, for each $1\leq j\leq n$.\\

We prove here some criteria for convergence of meromorphic quadratic differentials. As throughout this paper, for ease of notation we shall write these criteria only for the case of a single marked point $p$ of order $k$, but the results hold for any number of them. We list them in order of sophistication - the later criteria will depend on the previous ones. The final one was used in the proof of Theorem \ref{thm:main} (see \S9).\\

For the following lemma, one might find it useful to keep in mind this toy example:\\

\textit{Example.} Let $f_i:\mathbb{C}\to \mathbb{C}$ be a sequence of meromorphic functions such that in an open set $U$ that is the complement of the closed ball $\overline{B_{R_0}(0)}$,  we have that $f_i(z) = \frac{\phi_i(z)}{z^n}$ for some fixed $n\geq 1$, where $\{\phi_i(z)\}_{i\geq 1}$ are holomorphic functions on $U$ that form a normal family. Then for any $R >R_0$ the values that $f_i$ take on the circle $\partial B_{R}(0)$ are uniformly bounded (independent of $i$). By the maximal principle, the $f_i$s  are uniformly bounded on $B_{R}(0)$, and hence after passing to a subsequence, they converge uniformly on compact sets to a meromorphic function $f$.\\

\begin{lem}[Criterion 1]\label{lem:conv1} Let $\{q_i\}_{1\leq i <\infty}$ be a sequence of meromorphic quadratic differentials in $\mathcal{Q}_m(\Sigma)$ having the local expressions
\begin{equation*}
q_i(z) = \frac{\phi_i(z)}{z^k}dz^2
\end{equation*}
in a fixed coordinate neighborhood $(U,z)$ around the pole $p$. If the holomorphic functions $\{\phi_i\}_{1\leq i <\infty}$ form a normal family and converge to a holomorphic function non-vanishing at $0$ then $q_i\to q\in \widehat{\mathcal{Q}}_m(\Sigma)$ after passing to a subsequence. 
\end{lem}

\begin{proof}
Let $V \subset U$ be the radius-$1/2$ disk containing $p$,  via the conformal chart $\phi$ that takes $U$ to the unit disk. 
The restriction of each $q_i$ to the compact subsurface $\Sigma^\prime = \Sigma\setminus V$ is holomorphic in its interior. Our goal is to show that they are locally uniformly bounded, and it is enough to show that their $L^1$-norms (or equivalently, the $q_i$-areas of $\Sigma^\prime$) have a uniform bound.\\

Fix an arbitrary conformal metric on $\Sigma^\prime$ and $\delta>0$ and let $D_\delta$ is a radius-$\delta$ disk on $\Sigma^\prime$. Then we have:\\

\textit{Claim. $\lVert q_i\rVert_{L^1(\Sigma^\prime)} \leq C\lVert q_i\rVert_{L^1(D_\delta)}$, where $C$ is a constant depending on $\Sigma^\prime$ and $\delta$.}\\

\textit{Proof.} The proof of this is identical to that of Lemma 12.1 of \cite{Dum2}. By rescaling, it suffices to prove that for \textit{unit-norm} quadratic differentials, $\lVert q_i\rVert_{L^1(D_\delta)}\geq 1/C$ for some constant $C$. We proceed by contradiction: if not, we obtain a sequence of unit-norm quadratic differentials on $\Sigma^\prime$ such that the $q_i$-area of $D_\delta$ goes to zero. By compactness of unit-norm quadratic differentials, there is a convergent subsequence, and the limiting quadratic differential has unit norm. However it vanishes on $D_\delta$, and hence, by holomorphicity, is identically zero on the surface, which is a contradiction.\\

By the above claim it is enough to show that the $q_i$-area of \textit{some} radius-$\delta$ disk is uniformly bounded (independent of $i$). We choose the metric, and $\delta$, such that the annular region $U\setminus V$ contains such a $\delta$-disk $W$. The  sequence $\{q_i\}_{i\geq 1}$ then restricts to a uniformly bounded family on $W$ by assumption, and this completes the proof. The final assumption of non-vanishing at $0$ ensures that the limiting meromorphic quadratic differential has a poles of the same order $k$ at $p$, so the limit lies in the space $\widehat{\mathcal{Q}}_m$.
\end{proof}

In what follows, we shall fix a sequence $\{q_i\}_{1\leq i <\infty}$ be a sequence of meromorphic quadratic differentials in $\widehat{\mathcal{Q}}_m(\Sigma)$  as in Criterion 1 above, together with a coordinate neighborhood $U$ around the pole $p$. In what follows, we shall implicitly identify $U$ with the disk $\mathbb{D}$ in $z$-coordinates.

\begin{lem}[Criterion 2]\label{lem:conv2} Let there exist a fixed meromorphic quadratic differential $q_0$ in $\mathbb{C}$ with a pole of order $k$,  and a sequence of univalent conformal maps $f_i:(U, z) \to (\mathbb{C}, 0)$ where $i\geq 1$ such that the restriction of $q_i$ to $U$ is the pullback of $q_0$ via $f_i$. Moreover, assume that there is a uniform bound on the derivatives:
\begin{equation}\label{eq:db}
c < \left| f_i^\prime (0) \right| < C
\end{equation}
for all $i$, for some fixed reals $c, C$.\\
Then $q_i\to q\in \widehat{\mathcal{Q}}_m(\Sigma)$ after passing to a subsequence. 
\end{lem}

\begin{proof} One applies Criterion 1: it is a standard fact that the derivative bound (\ref{eq:db}) implies that the family of univalent maps is normal, and so are the local expressions of the pullback differentials. Moreover, by Lemma \ref{lem:pullb} and the lower bound on the derivatives above,  the leading order term (coefficient of $1/z^n$) does not vanish in the limit, and hence the limiting differential $q$ has the pole of correct order.
\end{proof}

As in \S3.3 a planar end $\mathcal{P}_H$ can be thought of as the restriction of a ``standard" holomorphic quadratic differential on $\mathbb{C}$ to a neighborhood of $\infty$, or by inversion, the restriction of a meromorphic quadratic differential on $\mathbb{C}$  (see (\ref{eq:stan})), to a neighborhood $U_H$ of the pole at $0$. Note that the neighborhoods $U_H$ shrink down to $0$, at a controlled rate (Lemma \ref{lem:uh}).\\

In what follows, we shall consider, as earlier in the paper, a sequence $H_i\to \infty$ such that 
\begin{equation}\label{eq:shrink}
dist(0,\partial U_{H_i}) = O(r_i)
\end{equation}
where $r_i = 2^{-i}$ (see Lemma \ref{lem:uhi}).

\begin{lem}[Criterion 3]\label{lem:conv3}  Let there be the same setup as in Criterion 2, except that instead of a derivative bound (\ref{eq:db}) we have that $f_i$ maps the disk $B_{r_i} = \{z\in U | \left|z\right| \leq r_i\}$ to $U_{H_i}$ for the sequence $H_i\to \infty$. Then $q_i\to q\in \widehat{\mathcal{Q}}_m(\Sigma)$ after passing to a subsequence. 
\end{lem}
\begin{proof} One only needs to prove the derivative bound (\ref{eq:db}) and apply Criterion 2. Recall the following growth theorem  concerning a univalent conformal map $f:\mathbb{D}\to \mathbb{C}$ such that $f(0)=0$ (see, for example, Theorem 1.3 of \cite{Pom}):

\begin{equation*}
\left| f^\prime (0)\right|\frac{\left|z\right|}{(1+\left| z\right|)^2} \leq \left| f(z)\right| \leq \left| f^\prime (0)\right|\frac{\left|z\right|}{(1-\left| z\right|)^2}
\end{equation*}

Applying this to $f=f_i$, and $\left|z\right| = r_i$, we obtain from (\ref{eq:shrink}) that $\left|f(z)\right| = O(r_i)$ and the required bounds on $\left|f_i^\prime(0)\right|$ follow.

\end{proof}

\begin{lem}[Criterion 4]\label{lem:conv4}  Let there be the same setup as in Criterion 3, except that instead of $f_i$ mapping $B_{r_i}$ to $U_{H_i}$, it maps some simply connected set $V_i$ containing $0$ to $U_{H_i}$, where 
\begin{equation}\label{eq:dmb}
{d}r_i  \leq dist(0, \partial V_i) \leq {D}r_i
\end{equation}
for each $i\geq 1$ and some fixed reals $d, D$. (Here $dist(0,\partial U_{H_i}) = r_i$ as before.)\\
Then $q_i\to q\in \widehat{\mathcal{Q}}_m(\Sigma)$ after passing to a subsequence. 

\end{lem}
\begin{proof} As in the proof of the previous criterion, it shall suffice to show uniform derivative bounds of $f_i$ at $0$.\\

For any conformal map $f:\mathbb{D}\to \mathbb{C}$ the following holds for any $z\in \mathbb{D}$ (see Corollary 1.4 of \cite{Pom}):
\begin{equation*}
\frac{1}{4} \left(1 - \left|z\right|^2\right)\left|f^\prime(z)\right| \leq dist (f(z), \partial f(\mathbb{D})) \leq  \left(1 - \left|z\right|^2\right)\left|f^\prime(z)\right|
\end{equation*}
Consider the conformal map $f = \phi_{i}: \mathbb{D} \to V_i\subset \mathbb{C}$ that fixes $0$.\\
In particular, for $z=0$ we get by rearranging the above inequalities and using that $\phi_i(0)=0$ that
\begin{equation}\label{eq:shk2}
dist(0, \partial f(\mathbb{D})) \leq \left|\phi_i^\prime(0)\right| \leq 4 dist(0, \partial f(\mathbb{D}))
\end{equation}
Since $\phi_i(\mathbb{D}) = V_i$ and (\ref{eq:dmb}) holds this gives
\begin{equation}\label{eq:cr1}
{d}r_i\leq  \left|\phi_{i}^\prime(0)\right| \leq 4{D}r_i.
\end{equation}

Now, consider the univalent conformal map $g = f_i \circ \phi_{i}:\mathbb{D}\to U_{H_i}\subset \mathbb{C}$. Using the assumption that  $dist(0,\partial U_{H_i}) = r_i$ and the distortion estimate (\ref{eq:shk2}) we have by the same argument:
\begin{equation}\label{eq:cr2}
\left| g^\prime(0) \right| = O(r_i)
\end{equation}

But by the chain rule  $g^\prime(0) =f_i^\prime(0)\cdot \phi_{i}^\prime(0)$ so by (\ref{eq:cr1}) and (\ref{eq:cr2}) we see that $\left|f_i^\prime(0)\right|$ have uniform bounds (independent of $i$), and one can apply Criterion 2.
\end{proof}

\subsubsection*{Varying $\Sigma$}
The previous criteria were for a fixed surface $\Sigma$. We now record a criterion, for meromorphic quadratic differentials on a sequence of surfaces converging in $\mathcal{T}_g$, to converge in the subset $\widehat{\mathcal{Q}}_m$ of the total bundle $\mathcal{Q}_m$ over $\mathcal{T}_g$ consisting of meromorphic quadratic differentials of a pole of order \textit{exactly} k.\\

Namely, consider a sequence of meromorphic quadratic differentials $\{q_i\}_{i\geq 1}$ on underlying surfaces $\{\Sigma_i\}_{i\geq 1}$ in $\mathcal{T}_g$ that converge to $\Sigma$, that is, there exist $(1+\epsilon_i)$-quasiconformal maps $h_i:\Sigma_i \to \Sigma$ where $\epsilon_i\to 0$ as $i\to \infty$. Suppose also that there exist a neighborhood $U_i$ of the pole of $q_i$ such that $h_i(U_i) = U$  for the fixed disk $U$ on $\Sigma$.

\begin{lem}[Criterion $1^\prime$] \label{lem:conv5} Assume that in the sequence just described, the restriction of $q_i$  on $U_i$ is, in local coordinates via a conformal identification $c_i:(U_i,p)\to (\mathbb{D},0)$, given by: \\
\begin{equation*}
q_i(z) = \frac{\phi_i(z)}{z^k}dz^2
\end{equation*}
where $\phi_i(z)$ are univalent holomorphic functions that form a normal family that converges to a holomorphic function non-vanishing at $0$.\\
Then $q_i\to q\in \widehat{\mathcal{Q}}_m(\Sigma)$ after passing to a subsequence. 
\end{lem}
\begin{proof}
Via the $(1 + \epsilon_i)$-quasiconformal homeomorphisms $c_i \circ h_i^{-1}$, the functions $\phi_i$ can be pulled back to a family of locally integrable complex-valued functions on $U$. By the assumption that the $\{\phi_i\}$ are a normal family, by arguing as in the proof of Criterion 1, one shows that the  pullback quadratic differentials (which are measurable, but no longer holomorphic, sections of $K_\mathbb{C}^{\otimes 2}$ for the Riemann surface $\Sigma$) have bounded $L^2$ norms on every (fixed) compact set away from $p$, and hence converge after passing to a subsequence. Since $\epsilon_i\to 0$,  away from $p$ the limiting quadratic differential is weakly holomorphic, and hence by Weyl's Lemma, holomorphic. By the condition that the limit of $\phi_i$ is non-vanishing at $0$, the quadratic differential has a pole of order $k$ at $p$.
\end{proof}

The above lemma implies that appropriate versions of Criteria 2, 3 and 4 also hold for the case when the underlying Riemann surfaces vary and form a converging sequence. In particular, we have:

\begin{lem}[Criterion $4^\prime$] \label{lem:conv6} Assume that in the sequence just described, the restriction of $q_i$  on $U_i$ is the pullback by a univalent conformal map $f_i:U_i\to \mathbb{C}$ of a fixed meromorphic quadratic differential $q_0$ on $\mathbb{C}$. Moreover, $f_i$ maps the subdomain $V_i\subset U_i$  to $U_{H_i}\subset \mathbb{C}$, where
\begin{equation*}
dist(0,\partial U_{H_i}) = r_i
\end{equation*}
 and $V_i$ satisfies the uniform distance bounds
\begin{equation*}
{d}r_i  \leq dist(0, \partial V_i) \leq {D}r_i
\end{equation*}
as in (\ref{eq:dmb}). (Here $d, D>0$ are some constants independent of $i$.) \\
Then $q_i\to q\in \widehat{\mathcal{Q}}_m(\Sigma)$ after passing to a subsequence. 
\end{lem}

The above lemma is summarized as Lemma \ref{lem:convg} in \S9, and used in the proof of Proposition \ref{prop:step4}.

\bibliographystyle{amsalpha}
\bibliography{qmref}

\end{document}